\documentclass{article}
\usepackage{graphicx, amsmath, amsfonts, amssymb, amsthm} 
\usepackage{graphicx,color, amsthm, amsmath, amssymb, amsfonts, float, xcolor, tikz} 
\usepackage[a4paper, total={6.5in, 8.5in}]{geometry}
\usepackage{comment}
\usepackage{enumerate}
\usepackage{enumitem}

\usepackage{thmtools, thm-restate}

\usepackage[colorlinks=true,linkcolor=black,urlcolor=black,citecolor=black]{hyperref}
\usepackage[nameinlink]{cleveref}

\usepackage[title]{appendix}

\newtheorem{theorem}{Theorem}[section]
\newtheorem{lemma}[theorem]{Lemma}
\newtheorem{prop}[theorem]{Proposition}
\newtheorem{corollary}[theorem]{Corollary}

\newtheorem{definition}[theorem]{Definition}
\newtheorem{conjecture}[theorem]{Conjecture}
\newtheorem{remark}[theorem]{Remark}

\newtheorem{claim}[theorem]{Claim}
\newtheorem{problem}[theorem]{Problem}

\newcommand{\R}{\mathbb{R}}
\newcommand{\SO}{\operatorname{SO}}
\newcommand{\SE}{\operatorname{SE}(2)}

\newcommand{\Otwo}{\operatorname{O}(2)}
\newcommand{\Othree}{\operatorname{O}(3)}
\newcommand{\OPseudo}{\operatorname{O}(1,1)}
\newcommand{\Od}{\operatorname{O}(d)}
\newcommand{\E}{\operatorname{E}}

\newcommand{\GL}{\operatorname{GL}}
\newcommand{\Aff}{\operatorname{Aff}}

\newcommand{\rank}{\operatorname{rank}}
\newcommand{\triv}{\operatorname{triv}}
\newcommand{\graph}{\operatorname{graph}}
\newcommand{\Spec}{\operatorname{Spec}}

\tikzstyle{vertex}=[fill=black,circle,inner sep=0pt, minimum size=4pt]
\tikzstyle{edge}=[line width=1.5pt,black]

\title{Generalised Erdős distance theory on graphs}
\author{Sean Dewar\thanks{School of Mathematics, University of Bristol. E-mail: \texttt{sean.dewar@bristol.ac.uk}, \texttt{sam.mansfield@bristol.ac.uk}, \texttt{jonathan.passant@bristol.ac.uk}} \and
Nora Frankl\thanks{School of Mathematics and Statistics, The Open University and Alfr\'{e}d R\'{e}nyi Institute of Mathematics \\
E-mail: \texttt{nora.frankl@open.ac.uk}} \and
Samuel Mansfield$^*$ \and
Anthony Nixon\thanks{School of Mathematical Sciences, Lancaster University. E-mail: \texttt{a.nixon@lancaster.ac.uk}} \and
Jonathan Passant$^*$ \and
Audie Warren\thanks{Radon Institute for Computational and Applied Mathematics, Austrian Academy of Sciences. E-mail: \texttt{audie.warren@oeaw.ac.at}}}
\date{}

\begin{document}

\maketitle

\begin{abstract}
The famous Erd\H{o}s distinct distances problem asks the following: how many distinct distances must exist between a set of $n$ points in the plane? There are many generalisations of this question that ask one to consider different spaces and metrics, or larger structures of points. We bring these problems into a common framework using the concept of $g$-rigidity.
Specifically, if $G=(V,E)$ is a (hyper)graph, $g$ is a map assigning polynomial measurements to the edges of $G$ and $f_{g,G}(P^V)$ gives the set of $g$-distinct realisations of the $g$-rigid graph $G$, where vertices must lie in a point set $P$, our main results describe sharp lower bounds for the size of $\big|f_{g,G}(P^V)\big|$.
This allows us to obtain results for pseudo-Euclidean metrics, {$\ell_p$} metrics,
dot-product problems, matrix completion problems, and symmetric tensor completion problems.

In addition, we use the recent work of Alon, Buci\'c and Sauermann along with a simple colouring argument to prove that the number of $\| \cdot\|$-distinct realisations of a graph $G=(V,E)$ within a $d$-dimensional point set $P$ is at least $\Omega\left(\frac{|P|^{|V|-1}}{(\log |P|)^2} \right)$ for almost all $d$-norms. 
Our methods here also provide a short proof that the unit distance conjecture implies the pinned distance conjecture.
\end{abstract}

\section{Introduction}

In this paper, we describe a common framework for generalised Erd\H{o}s distance problems, their connections to rigidity theory and prove sharp bounds for a wide variety of settings. The Erd\H{o}s distinct distances problem is the following.

\begin{problem}[Erd\H{o}s \cite{erdos1946sets}]\label{prob:Erdos}
    How many distinct distances must exist between pairs from a finite set of points in the plane?
\end{problem}

Erd\H{o}s conjectured that the correct answer to this problem is $\frac{c|P|}{\sqrt{\log|P|}}$, which is realised by taking $P$ to be the $\sqrt{|P|} \times \sqrt{|P|}$ integer grid. This conjecture was essentially solved by Guth and Katz \cite{guth2015erdHos} in 2015, where they prove that any set $P$ defines at least $\frac{c|P|}{\log|P|}$ distinct distances.

An alternative way of stating the distinct distances problem, which will be more useful for us, is that of lower bounding the number of images of a polynomial map known as the \emph{measurement map}, defined as follows for a general graph. Let $G = (V,E)$ be a (simple) graph. The measurement map $f_G$ of $G$ (in $\mathbb{R}^2$) is the polynomial
\begin{equation}\label{eq:measurementmap}
    f_G:\left(\mathbb R^2\right)^{V} \rightarrow \mathbb R^{E}, \quad p = \Big( p(v) \Big)_{v \in V} \mapsto \Big(\|p(v) - p(w) \|^2 \Big)_{vw \in E} ,
\end{equation}
where $\|\cdot\|$ is the standard Euclidean norm.
This map takes input positions for each vertex in the real plane, and measures (squared) distances of edges. 
Such a position of the vertices in the plane is called a \textit{realisation} of $G$. 
In this formulation, the distinct distance problem asks for a lower bound on $\big|f_{K_2}(P^2)\big|$:

\begin{problem}
    Let $P$ be any finite set of points in the plane. Give a lower bound on the size of the image set $f_{K_2}(P^{2})$.
\end{problem}

For more recent results on the standard distinct distances problem, we refer the reader to the book of Sheffer \cite{Sheffer_2022}. See the book of Garibaldi, Iosevich and Senger \cite{garibaldi2011erdos} and references therein for the earlier work on the problem.

\subsection{The graphical Erd\H{o}s distance problem and rigidity}

The measurement map formulation for the Erd\H{o}s distance problem leads to a natural question: what happens if we replace $K_2$ with another graph?
Here we can interpret the size of the image set $f_{G}(P^{V})$ as the number of \textit{inequivalent} realisations of the graph $G$ within $P$, where two realisations of $G$ are equivalent if they have the same edge lengths (this is tautological to them having the same image under $f_G$). A generalisation of the Erd\H{o}s distance problem is then the following.

\begin{problem}\label{prob:maineuclidproblem}
    Let $G$ be a fixed graph, and $P$ be any finite set of points in the plane. Give a lower bound on $\big|f_G\left(P^{V}\right)\big|$.
\end{problem}

The special case when $G$ is the complete graph on three vertices corresponds to lower bounding the number of congruence classes of triangles defined by $P$, as studied by Solymosi and Tardos \cite{solymositardos} and Rudnev \cite{rudnev2012triangle}.
Rudnev's approach hit upon a key aspect of the triangle: rigidity.
Roughly speaking,
we say that a graph-realisation pair $(G,p)$ (known also as a \emph{framework}) is \emph{rigid} if there exists finitely many other frameworks $(G,q)$ modulo isometries of the plane with the same edge lengths (or equivalently, where $f_G(p) = f_G(q)$).
Since every realisation of $K_3$ is rigid,
Guth and Katz's method can be adapted to show that\footnote{In this paper we use Vinogradov notation, that is, for two quantities $X$ and $Y$ which depend on a natural number parameter $n$, we write $X \ll Y$ if there exists a constant $c>0$ such that for all sufficiently large $n$ we have $X \leq c Y$. If we have both $X \ll Y$ and $Y \ll X$ then we write $X = \Theta(Y)$.} $f_{K_3}(P^3) \gg |P|^2$,
which can be shown to be tight using the integer grid.

Iosevich and Passant studied the problem more generally in \cite{ip18}.
Their methods prove that, if every realisation of $G$ in $\mathbb{R}^2$ with non-zero edge lengths is rigid, then $f_G(P^V) \gg |P|^{|V|-1}$ holds for any point set $P \subset \mathbb{R}^2$.
This family of graphs can be characterised by the lack of any \emph{NAC-colourings} (No Almost Cycles): a red-blue edge colourings using both colours such that every cycle is either monochromatic or contains at least 2 red edges and at least 2 blue edges \cite{gls19}.

Iosevich and Passant erroneously claim that their proof holds for the larger family of graphs which are \emph{rigid in $\mathbb{R}^2$}: any graph $G$ where every generic realisation of $G$ in the plane describes a rigid framework.
In this paper we provide a correction to this proof which solves \Cref{prob:maineuclidproblem} for all graphs that are rigid in $\mathbb{R}^2$.
In fact, we generalise \Cref{prob:maineuclidproblem} even further by allowing the edge map $f_G$ to be defined with reference to polynomials other than the Euclidean distance between two points.
We will therefore work with the concept of `$g$-rigidity', as detailed in the next section.

With regards to `non-rigid' graphs, lower bounds for $\big|f_G(P^V)\big|$ have been given only in the special cases of $G$ being a hinge (a path with two edges) by Rudnev \cite{RudnevHinge}, and more generally a path of arbitrary length, by Passant \cite{passantChains}.

\subsection{Generalised \texorpdfstring{Erd\H{o}s}{Erdos} problems via \texorpdfstring{$g$}{g}-rigidity}

The concept of $g$-rigidity was introduced by Cruickshank, Mohammadi, Nixon and Tanigawa \cite{cruik}. This replaces the Euclidean distance function $(x_1-x_2)^2 + (y_1-y_2)^2$ in the definition of $f_G$ with a polynomial 
\begin{equation*}
    g: (\mathbb R^d)^k \rightarrow \mathbb R, \quad x = (x_1, \ldots, x_k) \mapsto g(x_1, \ldots, x_k)
\end{equation*}
Making such a change of definition allows us to capture many other phenomena given by collections of points than just distances -- for example $g$ could give the dot product of two points, or the (signed) volume of a tetrahedron given by four points in $3$-space.  

As hinted above, our switch to a general polynomial allows us to generalise from graphs to $k$-uniform hypergraphs, since the input of $g$ may consist of multiple points. Since hyperedges are not ordered, we assume that either $g$ is \textit{symmetric} with respect to the ordering of the points in $\mathbb R^d$, or that $g$ is \textit{anti-symmetric} with respect to these points. In the latter case, we fix some ordering of the points and always input the points to $g$ in ascending order.

Having made these definitions, we can now define the $g$-\textit{measurement map} $f_{g,G}$ for some polynomial $g$ and a $k$-uniform hypergraph $G = (V,E)$ as above, as
\begin{equation*}
    f_{g,G}:(\mathbb R^d)^V \rightarrow \mathbb R^E, \quad p \mapsto \Big(g\big(p(v_1), \ldots, p(v_k)\big): e=\{v_1,\dots, v_k\}\in E\Big).
\end{equation*}
We can then state our further generalised variant of \Cref{prob:maineuclidproblem}.

\begin{problem}[Generalised Erd\H{o}s graphical problem]\label{prob:fullygeneralproblem}
    Let $G= (V,E)$ be a fixed $k$-uniform hypergraph, let $g: (\mathbb R^d)^k \rightarrow \mathbb R$ be a polynomial map which is either symmetric or anti-symmetric with respect to the input points, and let $P$ be any finite set of points in $\mathbb{R}^d$. Give a lower bound on $\big|f_{g,G}\left(P^{V}\right)\big|$.
\end{problem}

Of course, without any further restrictions we cannot say anything -- simply imagine the trivial case of $g$ being a constant map. However, this is where the concept of $g$-rigidity is useful. In order to introduce this, we will need to consider group actions which stabilise $g$. That is, consider some affine transformation $\gamma: \mathbb R^d \rightarrow \mathbb R^d$ where
$$\gamma(p) = Ap + b,\quad  A \in \GL(d),\ b \in \mathbb R^d.$$
This affine transformation is said to be a \emph{$g$-isometry} if $g(\gamma(x_1),\ldots,\gamma(x_k)) = g(x_1,\ldots,x_k)$ for all $x_1,\ldots,x_k \in \mathbb{R}^d$. The set of all such affine transformations forms a subgroup $\Gamma_g$ of the affine group -- we call this the \emph{$g$-isometry group}.
This group $\Gamma_g$ plays the role that the Euclidean group $E(d)$ has in standard rigidity theory. 
That is, a hypergraph should be called \emph{$g$-rigid} if, heuristically, most realisations $p$ of $G$ give rise to a set $f_{g,G}^{-1}(f_{g,G}(p))$ consisting of a finite union of orbits of $\Gamma_g$. 
Precise definitions of $g$-rigidity will be given in \Cref{sec:introtog-rigidity} -- to avoid making the introduction too technical we skip them here.

Before stating our results, we need one further definition -- that of the energy of a point set with respect to our $g$-isometries. 
Let $V$ be a finite set (think of the vertex set of a hypergraph), 
$P \subset \mathbb{R}^d$ and $\Gamma$ be any closed subgroup of $\Aff (d)$.
We define the \emph{energy} set of the triple $(P,V, \Gamma)$ by
\begin{equation*}
    \mathcal{E}_{V,\Gamma}(P) := \Big\{(p,q) \in P^V\times P^V:  ~ \theta p := \big(\theta(p(v))\big)_{v \in V} = q \text{ for some }\theta \in \Gamma\Big\}.
\end{equation*}
In words, this is the collection of pairs of realisations of $G$ in $P$, which are within the same $\Gamma$ orbit. We are now able to state out first main result.

\begin{restatable}{theorem}{energyg}\label{thm:energyg}
    Let $G=(V,E)$ be a $g$-rigid $k$-uniform hypergraph with at least $d+1$ vertices and maximum degree $\Delta_G$,
    and let $P \subset \mathbb{R}^d$ be a finite set where for any hypersurface $C \subset \mathbb{R}^d$ with degree at most $\Delta_G \deg g$, the inequality $|P\cap C| \leq \frac{1}{100}|P|$ holds.
    Then
    \begin{equation*}
        |P|^{2|V|} \ll\Big| f_{g,G}\left(P^V\right) \Big| \Big|\mathcal{E}_{V,\Gamma_g}(P)\Big|.
    \end{equation*}  
\end{restatable}

This implies that good upper bounds on $\big|\mathcal{E}_{V,\Gamma_g}(P)\big|$ give good lower bounds on $\big|f_{g,G}(P^V) \big|$. Using this, we can determine a general bound for any choice of $g$, such that $G$ is $g$-rigid, and the point set $P$ is not trapped on a hypersurface.

\begin{corollary}\label{cor:generalbound}
    Let $G=(V,E)$ be a $g$-rigid $k$-uniform hypergraph with at least $d+1$ vertices,
    and let $P \subset \mathbb{R}^d$ be a finite set where for any hypersurface $C \subset \mathbb{R}^d$ with degree at most $\Delta_G \deg g$, the inequality $|P\cap C| \leq \frac{1}{100}|P|$ holds.
    Then
    \begin{equation*}
       \Big| f_{g,G}\left(P^V\right) \Big| \gg |P|^{|V|-d-1}.
    \end{equation*}  
\end{corollary}

We additionally prove the following result to provide a complete answer in some of the specific cases when $d=k=2$.

\begin{restatable}{theorem}{Hypergraph}\label{thm:Hypergraph2D}
    Let $g: \mathbb{R}^2\times\R^2 \rightarrow \mathbb{R}$ be a (anti-)symmetric polynomial, let $G=(V,E)$ be a connected graph and let $D$ be a positive integer.
    Let $C \subset \mathbb{R}^2$ be an irreducible algebraic curve with degree at most $D$,
    and let $P \subset C$ be a finite set.
    If the map $g|_{C^2}$ is not constant,
    then
    \begin{equation*}
       \Big| f_{g,G}\left(P^V\right) \Big| \geq \left( \frac{1}{D\cdot\deg(g)} \left\lfloor\frac{|P|-D\cdot\deg(g)}{|V| } \right\rfloor \right)^{|V|-1}.
    \end{equation*} 
    In particular,
    $\big|f_{g,G}\left(P^V\right)\big| \geq c |P|^{|V|-1}$ for some constant $c>0$ that is dependent only on $G$, $g$ and $D$.
\end{restatable}

\subsection{Results for specific \texorpdfstring{$g$}{g}-maps}\label{sec:SpecificExamplesOfInterest}

In this section, we give various corollaries of \Cref{thm:energyg} and \Cref{thm:Hypergraph2D}, solving specific cases of the graphical Erd\H{o}s problem.

\subsubsection{Euclidean distance}

One of the key steps in Guth and Katz's solution to \Cref{prob:Erdos} is proving that the following holds:
\begin{equation*}
    |\mathcal{E}_{[2],\SE}(P)| \ll |P|^{3} \log |P|.
\end{equation*}
When $|V| \geq 3$ the result of Guth and Katz give us that
\begin{equation}\label{eq:guthkatzenergy}
    |\mathcal{E}_{V,\SE}(P)| \ll |P|^{|V|+1}.
\end{equation}
This was a crucial observation for both Rudnev \cite{rudnev2012triangle}, when $|V| = 3$, then Iosevich and Passant when $|V| \geq 3$ \cite{ip18}.
Importantly, the logarithmic term vanishes when $|V| \geq 3$.
It is not too difficult to show that \cref{eq:guthkatzenergy} holds if we replace $\SE$ with $\E(2)$ (see \Cref{lem:rotations} below).
Combining this with \Cref{thm:energyg} and \Cref{thm:Hypergraph2D},
we expand the result of Iosevich and Passant to all rigid graphs in $\mathbb{R}^2$.

\begin{corollary}\label{cor:euclidean}
    Let $G=(V,E)$ be a graph with $|V| \geq 3$ that is rigid in $\mathbb{R}^2$,
    and let $f_G$ be the measurement map given in \cref{eq:measurementmap}.
    Then for any finite set $P \subset \mathbb{R}^2$, we have
    \begin{equation*}
        \Big|f_{G}\left(P^V\right) \Big|  \gg |P|^{|V|-1},
    \end{equation*}  
    and this bound is tight.
\end{corollary}

An easy way to see that this bound is indeed tight is to observe how $f_G$ behaves when the points $P$ are evenly spaced on a line.

\subsubsection{Pseudo-Euclidean metrics}

A \emph{pseudo-Euclidean space of signature $(k,d-k)$} is the space $\mathbb{R}^d$ equipped with a quadratic form $(x,y) \mapsto x^T M y$ such that $M$ is a symmetric matrix with $k$ positive eigenvalues and $d-k$ negative eigenvalues.
It follows from Sylvester's law of inertia that we may always change our basis (which corresponds to an isometry between quadratic forms) so that we only need to consider the quadratic forms 
\begin{equation*}
    (x,y) \mapsto \sum_{i=1}^{k} (x_i-y_i)^2   - \sum_{i=k+1}^d (x_i-y_i)^2.
\end{equation*}
When $d=2$, there are only two pseudo-Euclidean spaces we need to consider:
standard Euclidean space ($\mathbb{R}^2$ with the squared distance map), or $\mathbb{R}^2$ equipped with the quadratic form
\begin{equation}\label{eq:pseudo}
    g: \mathbb{R}^2 \times \mathbb{R}^2 \rightarrow \mathbb{R}, ~ (x,y) \mapsto (x_1 - y_1)^2 - (x_2 - y_2)^2.
\end{equation}
In this latter case,
the group of $g$-isometries is exactly the (1,1)-pseudo-Euclidean isometry group $\Gamma_g := \R^2 \rtimes \OPseudo$,
where $\OPseudo$ consists of all matrices of the form
\begin{align*}
  \begin{pmatrix}
        \delta_1 & 0 \\
        0 & \delta_2
    \end{pmatrix}
    \begin{pmatrix}
        \cosh t & \sinh t \\
        \sinh t & \cosh t
    \end{pmatrix},
    \qquad t \in \mathbb{R}, ~ \delta_1,\delta_2 \in \{-1,1\}.
\end{align*}
The group $\OPseudo$ consists of four connected components.
The connected component containing the identity matrix, which is itself a Lie group of the same dimension, is the group
\begin{equation*}
    \SO^+(1,1) := 
    \left\{
    \begin{pmatrix} 
        \cosh t & \sinh t \\
        \sinh t & \cosh t 
    \end{pmatrix}
    : t \in \mathbb{R} 
    \right\}.
\end{equation*}
Interestingly, the graphs that are $g$-rigid with $g$ being the polynomial described in \cref{eq:pseudo} are exactly the graphs that are rigid in $\mathbb{R}^2$ \cite{sw07}.

Rudnev and Roche-Newton \cite{rnr15} used the subgroup $\mathbb{R}^2\rtimes\SO^+(1,1)$ of $\OPseudo$ to prove that either $f_{g,K_2}\left(P^V\right) = \{0\}$ (which occurs exactly when all points lie on the diagonal lines), or 
\begin{equation*}
    \Big|f_{g,K_2}\left(P^V\right) \Big|  \gg \frac{|P|}{\log|P|}.
\end{equation*}
By applying a minor adaption to switch from $\mathbb{R}^2\rtimes\SO^+(1,1)$ to $\mathbb{R}^2\rtimes\OPseudo$ (\Cref{lem:pseudoeuclideanenergy} below) and applying \Cref{thm:energyg} and \Cref{thm:Hypergraph2D},
we expand this to all rigid graphs with $|V| \geq 3$.

\begin{corollary}\label{cor:pseudo}
    Let $G=(V,E)$ be a graph with $|V| \geq 3$ that is rigid in $\mathbb{R}^2$, 
    and let $g$ be the pseudo-Euclidean quadratic form given in \cref{eq:pseudo}. 
    Let $L_1,L_2$ be the lines $L_1 = \{ x \in \mathbb{R}^2 : x_1 = x_2\}$ and $L_2 = \{ x \in \mathbb{R}^2 : x_1 = -x_2\}$.
    Then for any finite set $P \subset \mathbb{R}^2$ with a point $x \in P$ where $|(P -x) \setminus L_i| \gg |P|$ for both $i \in \{1,2\}$,
    we have
    \begin{equation*}
       \Big| f_{g,G}\left(P^V\right) \Big|  \gg |P|^{|V|-1},
    \end{equation*}  
    and this bound is tight.
\end{corollary}

Again, we can observe that the lower bound is tight by restricting $P$ to be a set of points in a line, so long as the line is not a translated copy of $L_1$ or $L_2$.

\subsubsection{The \texorpdfstring{$\ell_p$}{lp} metric}

A natural generalisation of the Euclidean metric is to replace the exponents with some value $1 \leq p < \infty$ to create an $\ell_p$ metric.
We can simplify the $\ell_p$ metric when $p$ is an even integer by removing the absolute value and $p$-root to obtain a polynomial map:
\begin{equation}\label{eq:lp}
    g(x,y) = \sum_{i=1}^d (x_i-y_i)^p.
\end{equation}
When $p=2$, the above $g$ is exactly the squared distance map.
Any graph which is $g$-rigid with respect to the above $g$ polynomial is said to be \emph{rigid in $\ell_p^d$}.

These differing notions of rigidity can agree (for example, the complete graph $K_{2d}$ is rigid for any $1 < p < \infty$ \cite[Corollary 3.4]{Dewar2022}),
they can differ drastically.
For example, Kitson and Power proved that if a graph $G$ is rigid in $\ell_p^d$ for $p \neq 2$, then $|E| \geq d|V| - d$ (see \cite[Proposition 3.2]{kitsonpower}), but examples of rigid graphs in $\mathbb{R}^d$ exist with strictly less edges (for example, $K_{d+1}$).
Conversely, gluing any two graphs that are rigid in $\ell_p^d$ for $p \neq 2$ over a single vertex preserves rigidity in $\ell_p^d$, but this is obviously false in $\mathbb{R}^d$ when $d \geq 2$.

In fact, there is an important reason why $\ell_p$ rigidity often differs when $p=2$ and $p \neq 2$: the isometry group for the $\ell_p$ metric when $p \neq 2$ is substantially smaller than the isometry group for $d$-dimensional Euclidean space.
When $p \neq 2$, the $g$-isometries are exactly the affine maps $x \mapsto A x + b$ where $b \in \mathbb{R}^d$ and $A$ is any signed permutation matrix (every row has exactly one non-zero entry, which takes the value $1$ or $-1$).
This lets us determine the energy for our graphs very easily for even $p \geq 4$ (see \Cref{lem:energytranslationonly}).
As any graph $G=(V,E)$ with $2 \leq |V| \leq 2d-1$ is not rigid in $\ell_p^d$,
we obtain the following result using \Cref{thm:energyg}.

\begin{corollary}\label{cor:lphigher}
    Let $G=(V,E)$ be a graph that is rigid in $\ell_p^d$ for even $p \geq 4$ and $d\geq 1$,
    and let $g$ be the $\ell_p$ metric given in \cref{eq:lp} for the chosen $p,d$.
    Let $P \subset \mathbb{R}^d$ be a finite set where for any hypersurface $C \subset \mathbb{R}^d$ with degree at most $\Delta_G(p-1)$, the inequality $|P\cap C| \leq \frac{1}{100}|P|$ holds.
    Then
    \begin{equation*}
        \Big|f_{g,G}\left(P^V\right) \Big| \gg |P|^{|V|-1},
    \end{equation*}
    and this bound is tight.
\end{corollary} 

To see that the lower bound is tight, let $x \in \mathbb{R}^d$ and $P= \{ k x : 0 \leq k \leq n-1\}$ for some $n$:
as translations provide equivalent realisations,
we have
\begin{align*}
    f_{g,G}\left(P^V\right) \leq \Big| \left\{p \in P^V : p(v) = \mathbf{0} \text{ for some } v \in V \right\}  \Big| \leq |V||P|^{|V|-1} \ll |P|^{|V|-1}.
\end{align*}

It was previously shown by Kitson and Power \cite{kitsonpower} that a graph is rigid in $\ell_p^2$ if and only if it contains two edge-disjoint spanning trees.
When we combine this with \Cref{thm:Hypergraph2D} and \Cref{cor:lphigher},
we obtain the following full characterisation.

\begin{corollary}\label{cor:lp}
    Let $G=(V,E)$ be a graph containing 2 edge-disjoint spanning trees,
    and let $g$ be the $\ell_p$ metric given in \cref{eq:lp} for even $p \geq 4$ and $d=2$.
    Then for any finite set $P \subset \mathbb{R}^2$,
    we have
    \begin{equation*}
        \Big|f_{g,G}\left(P^V\right) \Big|  \gg |P|^{|V|-1},
    \end{equation*}  
    and this bound is tight.
\end{corollary}

\begin{remark}
    The $\ell_p$ variant of \Cref{prob:Erdos} was first studied by Garibaldi \cite{juliaGaribaldiThesis}.
    With our terminology,
    Garibaldi proved that
    \begin{equation*}
        \Big| f_{g,K_2}\left(P^2\right) \Big| \gg |P|^{4/5}.
    \end{equation*}
    This was improved to a lower bound of $\Omega(|P|^{6/7-\varepsilon})$ for any $\varepsilon >0$ by a joint Polymath project involving AlQady, Chabot, Dudarov, Ge, Juvekar, Kundeti, Kundu, Lu, Moreno, Peng, Speas, Starzycka, Steinthal and Vitko working under the pseudonym `Polly Matthews Jr.' \cite{matthews2022distinct}. 
    The exact lower bound for the distinct $\ell_p$ problem is still open.
    Our techniques do not provide any direction for solving this problem, as $K_2$ is not rigid in $\ell_p^2$ for $p \neq 2$.
\end{remark}

\subsubsection{The dot product and matrix completion}

For our next $g$ metric, we introduce \emph{semisimple graphs}:
a multigraph that allows for loops but not repeated edges/loops.
We say that a semisimple graph is \emph{locally completable in $\mathbb{R}^d$} if it is $g$-rigid with respect to the polynomial
\begin{equation}\label{eq:dot}
    g: \mathbb{R}^d \times \mathbb{R}^d \rightarrow \mathbb{R}, ~ (x,y) \mapsto  x \cdot y = \sum_{i=1}^d x_i y_i.
\end{equation} 
The rigidity of such graphs is motivated by the problem of completing a partially filled positive semidefinite matrix of a given rank; we refer the reader to \cite{singercucu} for more details and variants of this problem. The exact characterisation of semisimple graphs which are locally completable in $\mathbb{R}^d$ is unknown for $d \geq 2$. 

The $g$-isometries in this case are exactly the orthogonal group $\Od$.
Since $\Otwo$ is contained in the Euclidean group $\E(2)$,
we can apply \Cref{thm:energyg} and \Cref{thm:Hypergraph2D} to achieve the following result.

\begin{corollary}\label{cor:dot}
    Let $G=(V,E)$ be a semisimple graph with $|V| \geq 3$ that is locally completable in $\mathbb{R}^2$,
    and let $g$ be the dot product given in \cref{eq:dot} with $d=2$.
    Then for any finite set $P \subset \mathbb{R}^2$,
    we have
    \begin{equation*}
       \Big| f_{g,G}\left(P^V\right) \Big|  \gg |P|^{|V|-1},
    \end{equation*}  
    and this bound is tight.
\end{corollary}

A witness for the lower bound can be found by taking $P$ to be the set of points $(\cos 2\pi k/n, \sin 2\pi k/n)$ for $k \in \{0,\ldots,n-1\}$. 

We can actually extend this to 3-space through an adaptation of Guth and Katz's key result regarding rich sets; see \Cref{thm: GuthKatz} for a statement of Guth and Katz's result and \Cref{lem:energyrotationonly} for the adaptation to the group of rotations and reflections.

\begin{corollary}\label{cor:dothigher}
    Let $G=(V,E)$ be a semisimple graph with $|V| \geq 4$ that is locally completable in $\mathbb{R}^3$,
    and let $g$ be the dot product given in \cref{eq:dot} with $d=3$.
    Let $P \subset \mathbb{R}^3$ be a finite set where for any hypersurface $C \subset \mathbb{R}^d$ with degree at most $\Delta_G$, the inequality $|P\cap C| \leq \frac{1}{100}|P|$ holds.
    Then
    \begin{equation*}
       \Big| f_{g,G}\left(P^V\right) \Big| \gg |P|^{|V|-1},
    \end{equation*} 
    and this bound is tight.
\end{corollary}

Again, a witness for the lower bound can be found by taking $P$ to be the points $(\cos 2\pi k/n, \sin 2\pi k/n,0)$ for each $k \in \{0,\ldots,n-1\}$.

Combining \Cref{cor:dot,cor:dothigher} with the ideas of Singer and Cucuringu \cite{singercucu},
we obtain the following result regarding positive semidefinite matrices.

\begin{corollary}\label{cor:psdmatrix}
    Fix $r \in \{1,2,3\}$ and an integer $n \geq r+1$.
    Let $P \subset \mathbb{R}^r$ be a finite set;
    furthermore, if $r=3$,
    we also suppose that for any hypersurface $C \subset \mathbb{R}^r$ with degree at most $n$, the inequality $|P\cap C| \leq \frac{1}{100}|P|$ holds.
    Given the set of rank $r$ positive semidefinite $n\times n$ matrices
    \begin{equation*}
        M_n(P) := \left\{ X^T X : X = [x_1 ~ \cdots ~ x_n] , ~ x_1,\ldots , x_n \in P \right\}
    \end{equation*}
    then
    \begin{equation*}
        |P|^{n-1} \ll \big|M_n(P) \big| \ll |P|^{n},
    \end{equation*}
    and both bounds are tight.
\end{corollary} 

\begin{remark}\label{rem:dot}
    The recent result of Hanson, Roche-Newton and Senger \cite{hanson2023convexity} gives the best-known bound on the problem of distinct dot-products -- that is, when $g$ is the dot product given in \cref{eq:dot}, we have
    \begin{equation*}
        \Big|f_{g,K_2}\left(P^2\right)\Big| \gg |P|^{2/3+\frac{1}{3057}+\varepsilon}.
    \end{equation*}
    The reason for the notably lower bound compared to our results above or Guth and Katz's solution to the original Erd\H{o}s distinct distance problem is that $K_2$ is not locally completable in $\mathbb{R}^2$.
    This leaves Hanson, Roche-Newton and Senger at a major disadvantage, since they can no longer convert the problem into one solely regarding isometries.
\end{remark}

\subsubsection{Skew-symmetric bilinear forms}

The unique (up to scaling) skew-symmetric bilinear form for $d=2$ is given by the polynomial
\begin{equation}\label{eq:skewsymm}
    g: \mathbb{R}^2 \times \mathbb{R}^2 \rightarrow \mathbb{R}, ~ (x,y) \mapsto x_1 y_2 - x_2 y_1.
\end{equation}
The $g$-isometries in this case are exactly the matrices contained in $\OPseudo$.
It is a relatively easy exercise to show that a semisimple graph is $g$-rigid with respect to the above polynomial if and only if it is locally completable in $\mathbb{R}^2$ (hint: observe how the change in metric effects the Jacobian of $f_{g,G}$).
Using \Cref{thm:energyg} and \Cref{thm:Hypergraph2D} -- as well as a previously-mentioned energy result that also applies to subgroups of $\mathbb{R}^2\rtimes \OPseudo$ (\Cref{lem:pseudoeuclideanenergy}) --
we obtain an analogous result to \Cref{cor:dot}, with the caveat that our point set is not mostly contained in either of the diagonals.

\begin{corollary}\label{cor:skewsymm}
    Let $G=(V,E)$ be a semisimple graph with $|V| \geq 3$ that is locally completable in $\mathbb{R}^2$,
    and let $g$ be the skew-symmetric bilinear form given in \cref{eq:skewsymm}. 
    Let $L_1,L_2$ be the lines $L_1 = \{ x \in \mathbb{R}^2 : x_1 = x_2\}$ and $L_2 = \{ x \in \mathbb{R}^2 : x_1 = -x_2\}$.
    Then for any finite set $P \subset \mathbb{R}^2$ where $|P \setminus L_i| \gg |P|$ for $i \in \{1,2\}$,
    we have
    \begin{equation*}
       \Big| f_{g,G}\left(P^V\right) \Big|  \gg |P|^{|V|-1},
    \end{equation*}  
    and this bound is tight.
\end{corollary}

\begin{remark}
    Iosevich, Roche-Newton and Rudnev \cite{iosevich2015discrete} show that for $P\in \R^2$ not supported on a line through the origin, we have
    \begin{equation*}
        \Big|f_{g,K_2}\left(P^2\right)\Big| \gg |P|^{9/13}.
    \end{equation*}
    This lower bound is lower than the one given in \Cref{cor:skewsymm}. Similar to \Cref{rem:dot}: $K_2$ is not locally completable in $\mathbb{R}^2$, and hence is not $g$-rigid with respect to the skew-symmetric bilinear form given by \cref{eq:skewsymm}.
\end{remark}

\subsubsection{Symmetric tensor completion}

Given $k$ points $x(1),\ldots,x(k)$ with $x(i) = (x_j(i))_{j=1}^d \in \mathbb{R}^d$,
we define the symmetric polynomial map 
\begin{equation}\label{eq:sym_tensor_g}
g(x(1),\ldots,x(k)) = \sum_{j=1}^d x_j(1) \cdots x_j(k) .
\end{equation}
If $k=2$, then $g$ describes the dot product.
If $k \geq 3$, then the $g$-isometries here are either the signed permutation matrices if $k$ is even, or the permutation matrices if $k$ is odd;
in either case,
the $g$-isometry group $\Gamma_g$ is finite.
In such a case, we see that $\mathcal{E}_{V,\Gamma_g}(P) \ll |P|^{|V|}$ (\Cref{cor:finiteisom}).
Combining this with the natural upper bound of $|P|^{|V|}$ for the size of the set $f_{g,G}\left(P^V\right)$,
we are able to use \Cref{thm:energyg} to prove the following result for any dimension.
As far as the authors are aware, this is the first result of the same type as the Erd\H{o}s distinct distance problem that uses this specific metric.

\begin{corollary}\label{cor:tensor}
    Let $G=(V,E)$ be a $k$-uniform hypergraph for $k \geq 3$.
    Further suppose that $G$ is $g$-rigid with respect to the polynomial map $g$ given in \cref{eq:sym_tensor_g} for $d\geq 1$.
    Let $P \subset \mathbb{R}^d$ be a finite set where for any hypersurface $C \subset \mathbb{R}^d$ with degree at most $\Delta_G(k-1)$, the inequality $|P\cap C| \leq \frac{1}{100}|P|$ holds.
    Then
    \begin{equation*}
       \Big| f_{g,G}\left(P^V\right) \Big| = \Theta \Big( |P|^{|V|} \Big).
    \end{equation*}
\end{corollary} 

Similar to the dot product,
we can interpret the metric given in \cref{eq:sym_tensor_g} using symmetric tensors.
To see this,
choose $d$ points $p_1,\ldots,p_d \in \mathbb{R}^n$,
where $p_i = \big(p_i(j) \big)_{j\in [n]}$ for each $i \in [d]$.
Then the tensor
\begin{equation*}
    T = \sum_{i=1}^r p_i^{\otimes k} = (T_{i_1,\ldots,i_k})_{i_1,\ldots,i_k \in [n]}
\end{equation*}
is a real symmetric tensor of order $k$ with rank\footnote{The rank of an order $k$ tensor of $\mathbb{R}^n$ is the smallest number $r$ of $k$-tuples $(x_{i,1},\ldots,x_{i,k}) \in (\mathbb{R}^n)^k$ so that
\begin{equation*}
    T= \sum_{i=1}^r x_{i,1}\otimes \cdots \otimes x_{i,k}.
\end{equation*}} at most $d$,
and
\begin{equation*}
    T_{i_1,\ldots,i_k} = g\Big( p(i_1),\ldots, p(i_k) \Big)
\end{equation*}
for every $i_1,\ldots,i_k \in [n]$.

Using this observation, we can use \Cref{cor:tensor} to show that the number of order $k$ tensors of rank at most $r$ that can be generated using a fixed set of points in $\mathbb{R}^r$ is always roughly the same.

\begin{corollary}\label{cor:psdtnesor}
    Fix integers $r$ and $k \geq 3$.
    Let $P \subset \mathbb{R}^r$ be a finite set so that for any hypersurface $C \subset \mathbb{R}^r$ with degree at most $\frac{1}{k!}n^k(k-1)$, the inequality $|P\cap C| \leq \frac{1}{100}|P|$ holds.
    Given the set of symmetric tensors of $\mathbb{R}^n$ of order $k$ with  rank $r$ given by the set
    \begin{equation*}
        T_n(P) := \left\{ T = \sum_{i=1}^r p_i^{\otimes k} : p_i = \big(p_i(j) \big)_{j\in [n]} \text{ for each $i \in [r]$, where } p(j) = (p_i(j))_{i \in [d]} \in P \text{ for each } j \in [n] \right\}
    \end{equation*}
    then
    \begin{equation*}
        \big|T_n(P) \big| = \Theta \Big( |P|^{n} \Big).
    \end{equation*}
\end{corollary}

\subsection{Generalised \texorpdfstring{Erd\H{o}s}{Erdos} problems via typical norms}\label{subsec:intro:norm}

Another way to generalise \Cref{prob:maineuclidproblem} is to define the map $f_G$ using an alternative norm to the Euclidean norm intrinsically used to define $f_G$;
specifically, for a \emph{$d$-norm} $\| \cdot\|$ (any norm for $\mathbb{R}^d$),
we define
\begin{equation*}
    f_{\|\cdot\|,G}:(\mathbb R^d)^{V} \rightarrow \mathbb R^{E}, \quad p = \Big( p(v) \Big)_{v \in V} \mapsto \Big( \big\|p(v)-p(w)\big\| \Big)_{vw \in E}.
\end{equation*}
Unlike with the polynomial $g$ metrics, norms can be particularly poorly behaved, allowing for little to no algebraic geometry tools to be implementable.
We can side-step this issue if we instead study what happens when we choose a `typical' norm.
Here a property holds for a \emph{typical $d$-norm} if it holds for all but a meagre set of norms, with respect to the Baire space $B_d$ of $d$-norms equipped with the Hausdorff metric.

In \cite{Alon2025}, the concept of typical norms was used to simplify two famous Euclidean distance problems: the Erd\H{o}s distinct distance problem (\Cref{prob:Erdos}) and the \emph{Erd\H{o}s unit-distance problem}, which asks for the maximum amount of distance-$1$ pairs within a finite point set $P$.
It is a conjecture that the correct upper bound for the Erd\H{o}s unit-distance problem is $c_{\epsilon}|P|^{1 + \epsilon}$ for all $\epsilon>0$.
The best known upper bound is $c|P|^{4/3}$, which follows from the Szemer\'{e}di-Trotter theorem \cite{SpencerUnit}. 
The Euclidean norm appears to be highly special with regards to this problem, and in \cite{Alon2025} (and also, with a slightly weaker bound, in \cite{MATOUSEK}) the following result was proved.

\begin{restatable}[Alon-Buci\'{c}-Sauermann {\cite{Alon2025}}]{theorem}{ABS}\label{ABS}
    For most $d$-norms $||\cdot||$ with $d \geq 2$, the number of unit distances determined by a finite point set $P \subseteq \mathbb R^d$ is at most $\frac{d}{2}|P| \log |P|$. Furthermore, the number of distinct distances determined by $P$ is $(1-o(1))|P|$.
\end{restatable}

In this paper we study the variant of \Cref{prob:maineuclidproblem} in typical norms. Specifically, we prove the following.

\begin{restatable}{theorem}{mainNormedResult}\label{mainNormedresult}
    The following is true for most $d$-norms $\|\cdot\|$. 
    Let $G = (V,E)$ be a connected graph, and let $P$ be any finite set of points in $\mathbb R^d$.
    Then 
    \begin{equation*}
       \Big|f_{\|\cdot\|, G}\left(P^{V}\right)\Big| \gg \left(\frac{|P|}{(\log |P|)^2}\right)^{|V|-1},
    \end{equation*}
    where the explicit constant depends only on the graph $G$ and the norm $\|\cdot\|$.
\end{restatable}

This result is proved via a pinned distance in typical norms result (which itself we reinterpret  as a graph colouring problem) and makes use of \Cref{ABS} as a `black-box' result.
Our result is almost tight, since $\big|f_{\|\cdot\|, G}\left(P^{V}\right)\big| \ll |P|^{|V|-1}$ when $P$ is an evenly-spaced set within a straight line.

\subsection{Layout of paper}

The paper is organised as follows.
In \Cref{sec:graphicalgrigid},
we cover the necessary background material regarding $g$-rigidity before proving \Cref{thm:energyg} and \Cref{thm:Hypergraph2D}.
In \Cref{sec:prop2.8proof},
we prove the technical result \Cref{prop:fewCongClassesg} that is required in the previous section.
In doing so, we improve upon previously known geometric results regarding $g$-rigidity; in particular, \Cref{prop:genericrealisationnumber} is a strengthening of \cite[Proposition 4.8]{cruik}.
We next turn our attention to normed spaces in \Cref{sec:norm} where we prove \Cref{mainNormedresult}.
Our graph colouring technique for proving \Cref{mainNormedresult} is then adapted to show that the weak unit distance conjecture implies weak pinned distance conjecture (\Cref{euclideanpinned}).
In \Cref{sec:flexible} we showcase that solving \Cref{prob:maineuclidproblem} becomes extremely difficult when we allow our graphs to be not rigid.
Simply put: if we cannot relate the problem back to one about isometries, solving \Cref{prob:maineuclidproblem} is linked to the uniform boundedness conjecture for rational points (\Cref{conj:uniformbounded}), an open problem in arithmetic geometry which generalises Faltings' theorem \cite{Faltings1983,Faltings1984}.
We conclude the paper in \Cref{sec:conclude} with some ideas for potential future research on the topic.

\subsection{Notation and terminology}

In this article hyperedges will be unordered (so $(x,y,z)$ and $(y,z,x)$ are the same edge) and we do allow for repeated vertices within a hyperedge; for example, $e=(x,x,y)$ would be an acceptable hyperedge for a 3-uniform hypergraph.
We do not, however, allow for repeated hyperedges within a hypergraph,
since this would correspond to repeating constraints.
We reserve the term `graph' for a 2-uniform hypergraph with no loops (equivalently, every graph is simple) and opt for `semisimple graph' for general 2-uniform hypergraphs (equivalently, a multigraph with loops but no repeated edges or loops).

Unless stated otherwise, we reserve $d$ to be the dimension of the space, $k$ to be the uniformity of the hypergraphs in question, and $g$ to be a symmetric or anti-symmetric polynomial $g: (\mathbb{R}^d)^k \rightarrow \mathbb{R}$.
The zero vector of $\mathbb{R}^d$ is denoted by $\mathbf{0}$.

We will use the standard notation $O$ and $\Omega$ for upper and lower bounds up to absolute constants, along with the symbols $\ll,\gg$, respectively. The symbol $\Theta$ means both $O$ and $\Omega$ hold.
The variable given in all asymptotic notation is the size of the point set in $\mathbb{R}^d$ (e.g., $|P|$).

For any positive integer $n$, we fix $[n] := \{1,\ldots,n\}$.

\section{Graphical Erd\H{o}s problem with algebraic constraints}\label{sec:graphicalgrigid}

\subsection{\texorpdfstring{$g$}{g}-rigidity preliminaries} \label{sec:introtog-rigidity}

We begin this section by formally defining the generalised model of $g$-rigidity introduced in \cite{cruik}.
Let $G=(V,E)$ be a $k$-uniform hypergraph.
An ordered pair $(G,p)$ consisting of $G$ and a point-configuration $p\in (\mathbb{R}^d)^{V}$ is called a {\em $d$-dimensional framework} or a {\em framework in $\mathbb{R}^d$}.
Suppose we are given a $k$-uniform framework $(G,p)$ and a polynomial map $g:(\mathbb{R}^{d})^k\rightarrow \mathbb{R}$.
The {\em $g$-measurement map} of $G$ is defined as a polynomial map $f_{g,G}:(\mathbb{R}^d)^{V}\rightarrow \mathbb{R}^E$ that sends $p$ to the list of the $g$-values of the tuples $(p(v_1), p(v_2), \dots, p(v_k))$ over the hyperedges $\{v_1,\ldots,v_k\}$ in $E$, i.e., 
\begin{equation*}
    f_{g,G}(p):=\Big(g\big(p(v_1), \ldots, p(v_k)\big): e=\{v_1,\dots, v_k\}\in E\Big).
\end{equation*}
We will assume that either $g$ is {\em symmetric}  with respect to the ordering of the points or that $g$ is {\em anti-symmetric} and that $v_1,\dots, v_k$ are put in increasing order in $g(p(v_1),\dots, p(v_k))$, assuming a fixed total order on the vertices of $G$. This assumption ensures that $f_{g,G}$ is well-defined. We say that two hyper-frameworks $(G,p)$ and $(G,q)$ are \emph{equivalent}, denoted $(G,p) \sim (G,q)$, if $f_{g,G}(p)=f_{g,G}(q)$. Such equivalences can be trivial, for example in the usual distance case one framework may arise from the other by an isometry. We next formalise this.

Suppose the general affine group ${\rm Aff}(d)$ acts on $\mathbb{R}^d$ by $\gamma\cdot x= Ax+t$ for $x\in \mathbb{R}^d$ and each pair $\gamma=(A,t)$ of $A\in {\rm GL}(d)$ and $t\in \mathbb{R}^d$.
The action of ${\rm Aff}(d)$ on $(\mathbb{R}^d)^{V}$ is also defined by $(\gamma \cdot p)(v)=Ap(v)+t$, with $v\in V$, for any $\gamma=(A,t)\in {\rm Aff}(d)$ and $p\in (\mathbb{R}^d)^{V}$.
Then the induced action on a polynomial map $g:(\mathbb{R}^d)^k\rightarrow \mathbb{R}$ is given by
\begin{equation*}
    \gamma\cdot g(x_1,\dots, x_k)=g(\gamma^{-1}\cdot x_1,\dots, \gamma^{-1} \cdot x_k)
\end{equation*}
for $x_1,\dots, x_k\in \mathbb{R}^d$ and $\gamma\in {\rm Aff}(d)$.
We say that $\gamma$ is a \emph{$g$-isometry} if $g$ is invariant by the action of $\gamma$. 
The group of $g$-isometries, here denoted by $\Gamma_g$, is a closed subgroup of ${\rm Aff}(d)$.
Thus $\Gamma_g$ is a Lie group by the closed subgroup theorem.
We define $\mathfrak{g}_g$ to be the Lie algebra of $\Gamma_g$,
i.e., the tangent space for $\Gamma_g$ at the identity map $I_d$. 

\begin{definition}
    We say that $(G,p)$ is {\em locally $g$-rigid} if there is an open neighbourhood $N$ of $p$ in $(\mathbb{R}^d)^{V}$ (with respect to the Euclidean topology) such that
    for any $q\in f^{-1}_{g,G}(f_{g,G}(p))\cap N$ there is $\gamma \in \Gamma_g$ such that $q=\gamma \cdot p$.
\end{definition}

Local $g$-rigidity is difficult to work with, so \cite{cruik} introduced a first order-variant that we now describe.

Observe that $\Gamma_g$ acts on $(\mathbb R^d)^V$ via $(\gamma\cdot p)(v)=  \gamma \cdot (p(v))$ for $v \in V$. 
This induces a map from $\mathfrak{g}_g \times (\mathbb R^d)^V $ to $(\mathbb R^d)^V$, denoted $(\dot \gamma, p) \mapsto \dot \gamma \cdot p$, whose image lies in the right kernel of the Jacobian matrix ${\rm J} f_{g,G}(p)$.
Then $\dot{p}\in (\mathbb{R}^d)^V$ defined by 
$\dot{p}=\dot{\gamma} \cdot p$ for $\dot{\gamma} \in \mathfrak{g}_g$ is called a {\em trivial infinitesimal $g$-motion} of $(G,p)$.
For $p\in (\mathbb{R}^d)^V$, define the space of trivial 
infinitesimal $g$-motions of $(G,p)$ to be
$$\triv_g(p):=\{\dot{\gamma}\cdot p: \dot{\gamma}\in \mathfrak{g}_g\}\subseteq \ker {\rm J} f_{g,G}(p).$$

\begin{definition}
    We say that $(G,p)$ is {\em infinitesimally $g$-rigid} if 
 $\triv_g(p)=\ker {\rm J} f_{g,G}(p)$.
\end{definition}

Importantly, infinitesimal $g$-rigidity implies local $g$-rigidity.
To see this, we need the following strengthening of \cite[Proposition 4.2]{cruik} that can be established with an identical proof.
We recall that, given a polynomial map $f:\mathbb{R}^n \rightarrow \mathbb{R}^m$,
we define $p \in \mathbb{R}^n$ to be \emph{generic (with respect to $f$)} if the coordinates of $p$ are algebraically independent over $\mathbb{Q}(f)$,
the extension of $\mathbb{Q}$ by the coefficients of the polynomial equations defining $f$. The following proposition shows that generic points always give regular values of polynomial maps. Its proof is immediate - the rank of a matrix being lower than its maximal value is described by zeroes of polynomials; the determinants of the appropriate minors.

\begin{prop}\label{prop:adaptedgeneric}
    Let $f: \mathbb{R}^n \rightarrow \mathbb{R}^m$ be a polynomial map.
    If $p \in \mathbb{R}^n$ is generic, then for each $q \in f^{-1}(f(p))$ the rank of ${\rm J} f (q)$ is maximal.
\end{prop}

The following result (originally given in \cite{cruik} for polynomial maps with rational coefficients) is a generalisation of a classical result of Asimow and Roth \cite{AsimowRothI}.

\begin{prop}\label{prop:asimowroth1}
    Let $G$ be a $k$-uniform hypergraph and let $(G,p)$ be a framework in $\mathbb{R}^d$.
    Then the following properties hold:
    \setlength{\parskip}{0pt}
    \begin{enumerate}[label=(\roman*)]
        \item If $(G,p)$ is infinitesimally $g$-rigid, then it is locally $g$-rigid.
        \item If $p$ is generic, then $(G,p)$ is infinitesimally $g$-rigid if and only if it is locally $g$-rigid.
    \end{enumerate}
\end{prop}

\begin{proof}
    The first part is \cite[Proposition 4.4]{cruik}.
    The second part follows the same proof given in \cite[Proposition 4.4]{cruik} by switching out \cite[Proposition 4.2]{cruik} for \Cref{prop:adaptedgeneric}.
\end{proof}

It follows from \Cref{prop:asimowroth1} that $g$-rigidity is a generic property: either every generic $d$-dimensional realisation is locally $g$-rigid, or no generic $d$-dimensional realisation is locally $g$-rigid.
This motivates the following definition of $g$-rigid.

\begin{definition}
    A $k$-uniform hypergraph $G$ is \emph{$g$-rigid} if every generic $d$-dimensional framework $(G,p)$ is locally $g$-rigid.
\end{definition}

\subsection{Proof of \texorpdfstring{\Cref{thm:energyg}}{first main theorem}}

In this subsection we will prove \Cref{thm:energyg}, which we restate here for convenience.

\energyg*

For a $k$-uniform hypergraph $G=(V,E)$ with rank $r$ in the $g$-rigidity matroid, 
consider the $|E| \times d|V|$ Jacobian matrix ${\rm J} f_{g,G}(p)$ as a symbolic matrix with variable $p$.
For every  $r \times r$ square matrix $M(p)$ of ${\rm J} f_{g,G}(p)$, the map $p \mapsto \det M(p)$ is a polynomial.
We now define $\Phi(G)$ to be the set of all such non-zero polynomials.
Since $r$ is the rank of $G$ in the $g$-rigidity matroid, the set $\Phi(G)$ is not empty.
Now, for every ordered set $X = \{v_0,v_1,\ldots,v_d\} \subset V$,
we define the polynomial 
\begin{equation*}
    f_X(p) = \det
    \begin{pmatrix}
        p(v_0) & p(v_1) & \ldots & p(v_d) \\
        1 & 1 & \ldots & 1
    \end{pmatrix},
\end{equation*}
and we fix $A(G)$ to be the set of all such polynomials.
We notice that a realisation $p$ of $G$ is \emph{$g$-regular} (i.e., ${\rm J} f_{g,G}$ has maximal rank at $p$) if and only if $f(p) \neq 0$ for some $f \in \Phi(G)$,
and furthermore that $p$ is \emph{affinely spanning} (i.e., the affine span of the set $\{p(v) :v \in V\}$ is $\mathbb{R}^d$) if and only if $f(p) \neq 0$ for some $f \in A(G)$.

\begin{lemma}\label{lem:allflexcurveg}
    Let $G=(V,E)$ be a $k$-uniform hypergraph with maximum degree $\Delta_G$.
    Suppose there exists a set $P \subset \mathbb{R}^d$ such that every injective realisation of $G$ in $P^V$ is not $g$-regular or not affinely spanning.
    Then at least $|P| - |V| + 1$ points (allowing for infinitely many points if $P$ is not finite) in $P$ are contained in a hyperplane or a hypersurface of degree at most $\Delta_G \deg g$.
\end{lemma}

\begin{proof}
    Without loss of generality, we may suppose that $P$ is maximal, in that for any point $z \in \mathbb{R}^d \setminus P$, there exists a $g$-regular, injective and affinely spanning realisation $q \in (P \cup \{z\})^V$ of $G$.    
    It follows that there is an injective realisation $p\in P^{V}$ and a vertex $u\in V$ which witnesses the maximality of $P$ in the following sense: there exists some $x \in \mathbb{R}^d \setminus P$ so that the realisation $p_x$ of $G$ defined as
    \begin{align*}
        p_x(v) =
        \begin{cases}
            x &\text{if } v = u, \\
            p(v) &\text{otherwise}
        \end{cases}
    \end{align*}
    is $g$-regular and affinely spanning.
    In particular, there exists $f_1 \in \Phi(G)$ and $f_2 \in A(G)$ such that $f_1(p_x) \neq 0$, $f_2(p_x) \neq 0$, and either $f_1(p)= 0$ or $f_2(p)=0$.
    By considering the realisation $p_x$ for each $x \in \mathbb{R}^d$,
    we define the non-constant polynomial function $h: \mathbb{R}^d \rightarrow \mathbb{R}$ by setting $h(x) = f_1(p_x)f_2(p_x)$ for each $x \in \mathbb{R}^d$.
    We observe here that, when constructing $f_1(p_x)$ (respectively, $f_2(p_x)$), at most $\Delta_G$ rows (respectively, a single column) of the defining matrix can contain variables corresponding to the vertex $u$.
    Hence, $h$ is a polynomial of degree at most $\Delta_G(\deg g-1) +1 \leq \Delta_G \deg g$ (as each partial derivative of $g$ has degree at most $\deg g -1$).
    We conclude the proof by noting that the zero set of $h$ contains the set $P \setminus\{p_v:v \in V \setminus \{u\} \}$, since by assumption each point in this set gives either $f_1(p_x)=0$ or $f_2(p_x)=0$.
\end{proof}

\begin{lemma} \label{lem:denseInfRigidg}
    Let $G=(V,E)$ be a $k$-uniform hypergraph,
    and let $P \subset \mathbb{R}^d$ be a finite set where $49|P|> 100|V|$.
    Suppose further that for any hypersurface $C \subset \mathbb{R}^d$ with degree at most $\Delta_G \deg g$, the inequality $|P\cap C| \leq \frac{1}{100}|P|$ holds.
    Then there are at least $|P|^{|V|}/2^{|V|}$ $g$-regular affinely spanning realisations of $G$ in $P^V$.
\end{lemma}

\begin{proof}
    By \Cref{lem:allflexcurveg},
    there exists a $g$-regular affinely spanning injective realisation $q \in P^V$ of $G$.
    Fix $X \subset P^V$ to be the set of $g$-regular affinely spanning realisations of $G$ in $P^V$.
    Label the vertices of $G$ by $v_1,\ldots,v_{|V|}$.
    We now inductively construct the sets $S_0,\ldots, S_{|V|} \subset P^V$ as follows.
    We fix $S_0 = \{q\}$,
    and for each $j \in [|V|]$,
    we fix
    \begin{equation*}
        S_j = \left\{ p' \in X : \text{ there exists }p\in S_{j-1}  \text{ such that } p'(v_i) = p(v_i) \text{ for each } i \neq j \right\}.
    \end{equation*}
    It now suffices to show that $|S_j| \geq (99|P|/100-|V|)|S_{j-1}|$ for each $j \in [|V|]$,
    as (using that $49|P|> 100|V|$) this implies
    \begin{equation*}
        |X| \geq |S_n| \geq \left(\frac{99|P|}{100}-|V| \right)^{|V|} = \left(\frac{50|P| + 49|P| - 100|V|}{100} \right)^{|V|} > \frac{|P|^{|V|}}{2^{|V|}}.
    \end{equation*}

    Fix $j \in [|V|]$. For each $p \in S_{j-1}$ and $x \in \mathbb R^d$, we define $p_{x}$ to be the realisation formed from $p$ by setting $p_x(v_j) = x$ and $p_x(v) = p(v)$ for all other $v$.
    Let $C$ be the algebraic set of points defined as 
    \begin{equation*}
        C := \left\{x\in \mathbb{R}^d: (G,p_x) \text{ is not $g$-regular or not affinely spanning} \right\}.
    \end{equation*}
    As $(G,p)$ is $g$-regular and affinely spanning, $p(v_j) \notin C$, and so $C$ is a proper algebraic set.
    Take any $f_1 \in \Phi(G)$ and $f_2 \in A(G)$ where $f_1(p) \neq 0$ and $f_2(p) \neq 0$, and let $h = f_1(p_x)f_2(p_x)$ be the $d$-variate polynomial formed by fixing all variables except the $d$ coordinates for the vertex $v_j$. Using a similar method to the one implemented in \Cref{lem:allflexcurveg},
    we see that $h$ has degree at most $\Delta_G \deg g$, and that zeroes of $h$ are the points $x \in \mathbb{R}^d$ where $p_x$ is a non-affinely spanning or non-regular realisation of $G$. 
    Hence, there are at least $99|P|/100-|V|$ points $x \in P$ such that $(G,p_x)$ is $g$-regular and affinely spanning and $x \neq p(v)$ for any $v \in V$.
    We now observe that, for any other realisation $p' \in S_{j-1}$ where $p' \neq p$,
    we have $p'_x \neq p_x$.
    Hence $|S_j| \geq (99|P|/100-|V|)|S_{j-1}|$.
\end{proof}

The following proposition will be a key part of our proof; we devote Section \ref{sec:prop2.8proof} to its proof.

\begin{restatable}{prop}{fewCongClassesg}\label{prop:fewCongClassesg}
    For any $k$-uniform hypergraph $G$ with at least $d+1$ vertices that is $g$-rigid, there exists a constant $C_G >0$ such that the following holds for any infinitesimally $g$-rigid framework $(G,p)$:
    given the set
    \begin{equation}\label{eq: InfRigidModEuclideanMotionsg}
        \mathcal{S} := \left\{ q \in (\mathbb{R}^d)^V : (G,p) \sim (G,q) \text{ and } (G,q) \text{ is infinitesimally $g$-rigid and affinely spanning } \right\} /~\Gamma_g,
    \end{equation}
    we have $|\mathcal{S}| \leq C_G$.
\end{restatable}

Given \Cref{prop:fewCongClassesg} we are now ready to prove \Cref{thm:energyg}.

\begin{proof}[Proof of \Cref{thm:energyg}]
Let $\mathcal{P}\subseteq P^V$ be the set of infinitesimally $g$-rigid and affinely spanning realisations of $G$, which has size at least $|P|^{|V|}/2^{|V|}$ by \Cref{lem:denseInfRigidg}.
For each point $\mathbf{t}$ in $\mathbb{R}^E$, we define its set of representations $\nu(\mathbf{t})$ as
    \[ \nu(\mathbf{t}) = \left\{ p \in P^V: f_{g,G}(p) = \mathbf{t} \text{ and } (G, p) \text{ is infinitesimally $g$-rigid and affinely spanning} \right\}.\]
Then
\begin{equation*}
    (|P|^{|V|})^2 \ll |\mathcal{P}|^2 = \left( \sum_{\mathbf{t} \in f_{g,G}(\mathcal{P})}|
    \nu( \mathbf{t})|\right)^2.
\end{equation*}
For each non-empty $\nu(\mathbf{t})$, there is some $p \in \mathcal{P}$ so that $f_{g,G}(p) = \mathbf{t}$. Select one such $p$ and fix it, calling this fixed element $p_{\mathbf{t}}$.
As all realisations in $\nu(\mathbf{t})$ have the same $g$-values (by definition), for any realisation $(G, q) \in \nu(\mathbf{t})$ we have that $(G,p_{\mathbf{t}}) \sim (G,q)$. Thus, modifying the definition of $\mathcal{S}$ from \eqref{eq: InfRigidModEuclideanMotionsg}, we have equivalence classes of infinitesimally $g$-rigid and affinely spanning frameworks under the action by the Euclidean group, defined as
\begin{equation*}
        \mathcal{S}(\mathbf{t}) := \left\{ q \in (\mathbb{R}^d)^V : (G,p_{\mathbf{t}}) \sim (G,q) \text{ and } (G,q) \text{ is infinitesimally $g$-rigid and affinely spanning} \right\} /~\Gamma_g.
\end{equation*}
By \Cref{prop:fewCongClassesg}, for all $\mathbf{t} \in \mathbb R^E$ we have $|\mathcal{S}(\mathbf{t})| \leq C_G$ for some constant $C_G$ independent of $|P|$. 
We define $\nu_i(\mathbf{t})$ to be $\nu(\mathbf{t})$ intersected with the $i^{th}$ equivalence class in $\mathcal{S}(\mathbf{t})$. Then we have
\begin{align*}
        |P|^{2|V|} \ll \left(\sum_{\mathbf{t}\in f_{g,G}(\mathcal{P})}|\nu(\mathbf{t})|\right)^2 = \left(\sum_{\mathbf{t}\in f_{g,G}(\mathcal{P})}\sum_{i=1}^{C_G}|\nu_i(\mathbf{t})|\right)^2.
    \end{align*}
For notational ease we assume that the support of each $\nu(\mathbf{t})$ is exactly $C_G$. This is not a problem as our argument uses $C_G$ as a uniform bound on the support of this second sum.
Upper bounding each $|\nu_i(\mathbf{t})|$ in this final line with the largest set of representatives, denoted by $|\nu_{\max}(\mathbf{t})|$, gives us
    \[ |P|^{2|V|} \ll \left(\sum_{\mathbf{t}\in f_{g,G}(\mathcal{P})} C_G|\nu_{\max}(\mathbf{t})|\right)^2.\]
We now apply Cauchy-Schwarz, to get
\begin{equation*}
        |P|^{2|V|} \ll C_G^2 \cdot\Big|f_{g,G}(\mathcal{P}) \Big|  \sum_{\mathbf{t}\in f_{g,G}(\mathcal{P})}|\nu_{\max}(\mathbf{t})|^2.
\end{equation*}
As $\mathcal{P} \subseteq P^V$ we have that $\big|f_{g,G}(\mathcal{P})\big| \leq \big|f_{g,G}\left(P^V\right)\big|$. As $C_G = O_{|V|}(1)$, the constant can be suppressed in the asymptotic notation. Putting this together gives
\begin{equation}\label{eq: F_GEnergyg}
    |P|^{2|V|} \ll \big|f_{g,G}\left(P^V\right)\big|  \sum_{\mathbf{t}\in f_{g,G}(\mathcal{P})}|\nu_{\max}(\mathbf{t})|^2.
\end{equation}
We observe now that
\begin{equation}\label{eq: F_GEnergy2g}
    \sum_{\mathbf{t}\in f_{g,G}(\mathcal{P})}|\nu_{\max}(\mathbf{t})|^2 \leq \big|\left\{(p,q) \in \mathcal{P}\times \mathcal{P}: \exists \theta \in \Gamma_g, ~ \theta p = q\right\}\big| \leq |\mathcal{E}_{V,\Gamma_g}(P)|.
\end{equation}
The desired result now follows from \eqref{eq: F_GEnergyg} and \eqref{eq: F_GEnergy2g}.
\end{proof}

\subsection{Resolution of `non-hypersurface' cases using energy}

We now apply \Cref{thm:energyg} to prove a variety of results under the hypothesis that the point set $P$ is not restricted to a low-degree hypersurface. That is, we prove various energy bounds for different isometry groups. Let $\Gamma$ be any subgroup of $\Aff (d)$. It is easy to see that for any $P \subset \mathbb{R}^d$,
\begin{equation*}
    |P|^{|V|} \leq |\mathcal{E}_{V,\Gamma}(P)|\leq |P|^{2|V|}.
\end{equation*}
To implement \Cref{thm:energyg}, we require much stronger bounds than this.
We begin with the following improvement on the upper bound for $\Gamma$ being the full affine group.

\begin{lemma}\label{thm:EasyEnergyBound}
    Let $V$ be a fixed finite set and let $\Gamma=\Aff(d)$.
    Then for any $P \subset \mathbb{R}^d$,
    we have
    \begin{equation*}
        |\mathcal{E}_{V,\Gamma}(P)| \ll |P|^{|V|+d+1}.
    \end{equation*}
\end{lemma}

\begin{proof}
Define $P^V_{i}$ to be the set of realisations $p \in P^V$ so that the affine span of $\{p_v : v \in V \}$ is exactly $i$-dimensional.
We observe the following: for all $\theta \in \Gamma=\Aff(d)$, there can be no pair $(p,\theta  p)$ where $p\in P^V_{i}$ and $\theta  p\in P^V_{j}$ where $i\neq j$. 
It follows that, if we define for any subsets $Q, Q' \subset P^V$ the set
\begin{equation*}
    \mathcal{E}_\Gamma(Q,Q') = \left\{(p,q) \in Q\times Q':  ~ \theta  p = q \text{ for some }\theta \in \Gamma\right\}
\end{equation*}
then,
\begin{equation}\label{eq:SpanningEnergyDecomp}
    \mathcal{E}_{V,\Gamma}(P) = \bigsqcup_{i=0}^d \mathcal{E}_\Gamma(P^V_{i},P^V_{i}).
\end{equation}
We now use the affine spanning condition to show that
\begin{equation}\label{eq:iDimBound}
    |\mathcal{E}_\Gamma(P^V_{i},P^V_{i})| \ll |P|^{|V|+(i+1)}.
\end{equation}
Then putting \cref{eq:iDimBound} into \cref{eq:SpanningEnergyDecomp} and using the union bound gives the result.
To prove \cref{eq:iDimBound} we look at one of the pairs $(p,q)$ in $\mathcal{E}_\Gamma(P^V_{i},P^V_{i})$, so $q=\theta p$.
As $q\in P_i^V$, there exist vertices $v_0,\ldots, v_i$ such that the affine span of $\{p(v_0),\ldots, p(v_i)\}$ is $i$-dimensional. 
For $w \in V \setminus \{v_0,\ldots,v_i\}$,
there exists $a_0,\ldots,a_i \in \mathbb{R}$ with $\sum_{j=0}^i a_i = 1$ so that $p(w) = \sum_{j=0}^i a_i p(v_i)$,
and hence $q(w) = \sum_{j=0}^i a_i q(v_i)$.
It follows that the choice of $p$ and the points $q(v_0), \ldots, q(v_i)$ entirely determines $\theta$. 
There are at most $|P|^{i+1}$ choices for $q(v_0)$ through $q(v_i)$.
Combining this with the $|P_i^V| \leq |P|^{|V|}$ choices for $p$ gives \cref{eq:iDimBound}.
\end{proof}

Since every $g$-isometry group $\Gamma_g$ is contained in $\Aff(d)$,
\Cref{cor:generalbound} now follows immediately from \Cref{thm:energyg} and \Cref{thm:EasyEnergyBound}.

At the opposite end of the spectrum, we have the following easy result.

\begin{lemma}\label{lem:energyfinitegroup}
    Let $V$ be a fixed finite set and let $\Gamma$ be a fixed finite subgroup of the group of $d$-dimensional affine transformations.
    Then for any $P \subset \mathbb{R}^d$,
    we have
    \begin{equation*}
        |\mathcal{E}_{V,\Gamma}(P)| \ll |P|^{|V|}.
    \end{equation*}
\end{lemma}

\begin{proof}
    For every $p \in P^V$, there exist at most $|\Gamma|$ realisations $q \in P^V$ such that $(p,q)\in \mathcal{E}_{V,\Gamma}(P)$.
    Hence,
    \begin{equation*}
        |\mathcal{E}_{V,\Gamma}(P)| \leq |\Gamma||P|^{|V|} \ll |P|^{|V|}.\qedhere
    \end{equation*}
\end{proof}

We can now use \Cref{lem:energyfinitegroup} to drastically improve \Cref{cor:generalbound}.

\begin{corollary}\label{cor:finiteisom}
    Let $G=(V,E)$ be a $g$-rigid $k$-uniform hypergraph with at least $d+1$ vertices,
    and let $P \subset \mathbb{R}^d$ be a finite set where for any hypersurface $C \subset \mathbb{R}^d$ with degree at most $\Delta_G \deg g$, the inequality $|P\cap C| \leq \frac{1}{100}|P|$ holds.
    If $|\Gamma_g|$ is finite,
    then
    \begin{equation*}
       \Big| f_{g,G}\left(P^V\right) \Big| = \Theta\left( |P|^{|V|} \right).
    \end{equation*}  
\end{corollary}

\begin{proof}
    It is immediate that $\big|f_{g,G}\left(P^V\right)\big| \leq |P^V| = |P|^{|V|}$,
    and $\big|f_{g,G}\left(P^V\right) \big| = \Omega\left( |P|^{|V|} \right)$ follows from \Cref{thm:energyg} and \Cref{lem:energyfinitegroup}.
\end{proof}

\subsubsection{Rich sets}

We are now going to exploit the concept of `rich sets' that was used by Guth and Katz to prove their landmark result regarding the Erd\H{o}s distinct distance problem.
Our variant of `rich sets' is slightly altered to work for any group of affine transforms, which we do via quotienting out repeated affine transforms.

\begin{lemma}\label{lem:hardpartg}
    Let $V$ be a fixed finite set,
    and let $\Gamma$ be a fixed subgroup of $\Aff (d)$.
    For each positive integer $t$ and any finite set $P \subset \mathbb{R}^d$,
    let
    \begin{equation*}
        \mathcal{R}_{\geq t}(P) := \{ \theta \in \Gamma : |P \cap \theta P| \geq t\}/ \sim_P,
    \end{equation*}
    where $\theta_1 \sim_P \theta_2$ if and only if $P \cap \theta_1(P) = P \cap \theta_2(P)$ and $P \cap \theta_1^{-1}(P) = P \cap \theta_2^{-1}(P)$.
    Then for any $P \subset \mathbb{R}^d$,
    we have
    \begin{equation*}
        \big|\mathcal{E}_{V,\Gamma}(P)\big| \ll \sum_{t= \max\{|V|-d-1, 1 \} }^{|P|}t^{|V|-1}\big|\mathcal{R}_{\geq t}(P)\big|.
    \end{equation*}
\end{lemma}

\begin{proof}
    Let us begin by defining the set of exactly $t$-rich transformations as
    \begin{equation*}
        \mathcal{R}_{=t}(P) := \{ \theta \in \Gamma : |P \cap \theta P| = t\}/ \sim_P.
    \end{equation*}
    For $t \geq |V|$, each $\theta \in \mathcal{R}_{=t}(P)$ contributes $\binom{t}{|V|}$ pairs of the form $(p,\theta p)$, with $p \in P^{V}$ being injective.
    Also, for every element $u \in V$,
    we gain $|V|-1$ pairs $(p,q) \in \mathcal{E}_{V,\Gamma}(P)$ from $\mathcal{E}_{V \setminus \{u\},\Gamma}(P)$ by setting $p(u) = p(v)$ and $q(u)=q(v)$ for some $v \in V \setminus \{u\}$.
    Accounting for repetitions when applying the previous method for each element $u \in V$,
    we see that (with a fixed arbitrary choice of $u$)
    \begin{align*}
        |\mathcal{E}_{V,\Gamma}(P)| &= \binom{|V|}{2} |\mathcal{E}_{V\setminus\{u\},\Gamma}(P)| + \sum_{t=|V|}^{|P|}\big|\mathcal{R}_{=t}(P)\big| \binom{t}{|V|} \\
        &\ll |\mathcal{E}_{V \setminus \{u\},\Gamma}(P)| + \sum_{t=|V|}^{|P|}\big|\mathcal{R}_{=t}(P)\big| t^{|V|}.
    \end{align*}
    (The final step uses that $\binom{t}{|V|} = O\left(t^{|V|}\right)$, as it will simplify the coming calculations.)     
    We now wish to switch to the sets $\mathcal{R}_{\geq t}(P)$, using the fact that
    \[\big|\mathcal{R}_{= t}(P)\big| = \big|\mathcal{R}_{\geq t}(P)\big| - \big|\mathcal{R}_{\geq(t+1)}(P)\big|.\]
    Thus for arbitrary $u$, we have (noting that $\mathcal{R}_{\geq |P|+1}(P)$ is empty)
    \begin{align*}    
        |\mathcal{E}_{V,\Gamma}(P)| &\ll |\mathcal{E}_{V \setminus \{u\},\Gamma}(P)| + \sum_{t=|V|}^{|P|}\big|\mathcal{R}_{=t}(P)\big| t^{|V|}\\
            &= |\mathcal{E}_{V \setminus \{u\},\Gamma}(P)| + \sum_{t=|V|}^{|P|}\left(\big|\mathcal{R}_{\geq t}(P)\big| - \big|\mathcal{R}_{\geq(t+1)}(P)\big|\right)t^{|V|},\\
            &= |\mathcal{E}_{V \setminus \{u\},\Gamma}(P)| + |V|^{|V|}\big|\mathcal{R}_{\geq |V|}(P)\big| + \sum_{t=|V|+1}^{|P|}(t^{|V|}-(t-1)^{|V|})\big|\mathcal{R}_{\geq t}(P)\big|.
    \end{align*}
    By using that $(t^{|V|}-(t-1)^{|V|}) \ll |V|t^{|V|-1}$ and pulling the constant $|V|$ into the asymptotic notation,
    we see that
    \begin{equation}\label{eq:recursive}
        \big|\mathcal{E}_{V,\Gamma}(P)\big| \ll |\mathcal{E}_{V \setminus \{u\},\Gamma}(P)| + \sum_{t=|V|}^{|P|}t^{|V|-1}\big|\mathcal{R}_{\geq t}(P)\big|.
    \end{equation}  
    As this is true for any subset of $V$,
    we can feed \eqref{eq:recursive} back into itself to see that for an arbitrary vertex set $W \subset V$ where $|W| = |V| - d - 2$ or $W = \emptyset$ if $|V| \leq d + 2$, we have
    \begin{align*}
        \big|\mathcal{E}_{V,\Gamma}(P)\big| 
        &\ll \big|\mathcal{E}_{W,\Gamma}(P)\big|+ \sum_{t= \max\{|V|-d-1, 1 \}}^{|V|}t^{t-1}\big|\mathcal{R}_{\geq t}(P)\big| + \sum_{t=|V|}^{|P|}|V|t^{|V|-1}\big|\mathcal{R}_{\geq t}(P)\big| \\
        &\ll \big|\mathcal{E}_{W,\Gamma}(P)\big|+ \sum_{t= \max\{|V|-d-1, 1 \}}^{|P|} t^{|V|-1}\big|\mathcal{R}_{\geq t}(P)\big|.
    \end{align*}
    If $W= \emptyset$ then $\big|\mathcal{E}_{W,\Gamma}(P)\big| = 1$ and the result holds.
    Suppose $W \neq \emptyset$.
    By \Cref{thm:EasyEnergyBound},
    we have that $\big|\mathcal{E}_{W,\Gamma}(P)\big| \leq |P|^{(|V|-d-2) + d+1} \leq |P|^{|V|-1}$.
    Therefore
    \begin{align*}
        \big|\mathcal{E}_{V,\Gamma}(P)\big| \ll |P|^{|V|-1} +  \sum_{t=|V|-d-1}^{|P|} t^{|V|-1}\big|\mathcal{R}_{\geq t}(P)\big|.
    \end{align*}
    The first term can now be removed, as $\big|\mathcal{E}_{V,\Gamma}(P)\big| \gg |P|^{|V|}$,
    to obtain the desired result.
\end{proof}

The first application of \Cref{lem:hardpartg} is to the Euclidean group $\E(2)$. We will use the following key result of Guth and Katz \cite[Proposition 2.5]{guth2015erdHos} which describes the sizes of $t$-rich groups.

\begin{theorem}[Guth and Katz \cite{guth2015erdHos}]\label{thm: GuthKatz}
    Fix $\Gamma = \SE$.
    Then there exists a constant $C>0$ so that the following holds:
    for any finite point set $P\subset \mathbb{R}^2$, and any $2 \leq t \leq |P|$, the size of $\mathcal{R}_{\geq t}(P)$ is bounded as follows:
    \[\big|\mathcal{R}_{\geq t}(P)\big| \leq C \frac{|P|^3}{t^2}.\]
\end{theorem}

A first consequence of this theorem is an energy bound for subgroups of $E(2)$.

\begin{lemma}\label{lem:rotations}
    Let $V$ be a fixed finite set with $|V| \geq 3$ and let $\Gamma$ be a subgroup of $\E(2)$.
    Then for any $P \subset \mathbb{R}^2$,
    we have
    \begin{equation*}
        |\mathcal{E}_{V,\Gamma}(P)| \ll |P|^{|V|+1}.
    \end{equation*}
\end{lemma}

\begin{proof}
    By combining \Cref{thm: GuthKatz} with \Cref{lem:hardpartg}, we see that
    \begin{equation*}
        |\mathcal{E}_{V,\SE}(P)| \ll \sum_{t=2}^{|P|}t^{|V|-1}\frac{|P|^3}{t^2} = |P|^3\sum_{t=2}^{|P|} t^{|V|-3} \ll |P|^3 \cdot |P|^{|V|-2} = |P|^{|V|+1}.
    \end{equation*}
    Therefore this upper bound holds for any subgroup of $\SE$. We now assume that $\Gamma$ is not a subgroup of $\SE$. Define 
    \begin{equation*}
        \theta_x := 
        \begin{pmatrix}
            1 & 0 \\
            0 & -1
        \end{pmatrix},
        \qquad 
        P' := P \cup \theta_x P.
    \end{equation*}
    That is, $\theta_x$ is a representative of the determinant $-1$ component of $E(2)$. We now note that:
    \begin{align*}
        (p,q) \in \mathcal{E}_{V, \Gamma}(P) \setminus  \mathcal{E}_{V,\SE}(P) \quad & \implies \exists \theta \in \Gamma \ \text{s.t.} \ p = \theta q \\ 
        & \implies \exists \theta_0 \in \SE \ \text{s.t.} \ p = \theta_0 \theta_x q \\
        & \implies (p, \theta_x q) \in \mathcal{E}_{V,\SE}(P')
    \end{align*}
    Hence,
    \begin{equation*}
        |\mathcal{E}_{V,\Gamma}(P)| \leq |\mathcal{E}_{V,\Gamma}(P')| \leq 2 |\mathcal{E}_{V,\SE}(P')| \ll |P'|^{|V|+1} \leq  2^{|V|+1} |P|^{|V|+1} \ll |P|^{|V|+1}.
    \end{equation*}
    This now concludes the proof.
\end{proof}

In fact,
we can also utilise \Cref{thm: GuthKatz} to obtain energy bounds for other groups which behave similarly.
We begin with the group $\Othree$. We will use the following result of Tao.

\begin{theorem}[Tao \cite{taoSphericalBlog}]\label{thm: sphereGuthKatz}
    Fix $\Gamma = \SO(3)$, let $Q\subset \mathbb{S}^2$ be a finite point set and take any $2 \leq t \leq |Q|$. Then there exists a constant $C>0$ such that  \[\big|\mathcal{R}_{\geq t}(Q)\big| \leq C \frac{|Q|^3}{t^2}.\]
\end{theorem}

\begin{lemma}\label{lem:energyspecialrotationonly}
    Let $V$ be a fixed finite set with $|V| \geq 4$ and let $\Gamma$ be a subgroup of $\SO(3)$.
    Then, for any $P \subset \mathbb{R}^3$,
    we have
    \begin{equation*}
        |\mathcal{E}_{V,\Gamma}(P)| \ll |P|^{|V|+1}.
    \end{equation*}
\end{lemma}

\begin{proof}
    If $P$ contains the point $(0,0,0)$ then 
    \begin{equation*}
        \Big|\mathcal{E}_{V,\Gamma}(P) \Big| = \Big|\mathcal{E}_{V, \Gamma} (P \setminus \{(0,0,0)\}) \Big| + \sum_{v \in V} \Big|\mathcal{E}_{V \setminus \{v\}, \Gamma} (P \setminus \{(0,0,0)\}) \Big|.
    \end{equation*}
    Hence, it suffices to suppose that $P$ does not contain the point $(0,0,0)$.    
    Similarly, if $P$ is contained in a line through the origin,
    then $\mathcal{E}_{V,\Gamma}(P) \subset \{(p, \pm p) : p \in P^V\}$,
    and so $|\mathcal{E}_{V,\Gamma}(P) | \leq 2|P|^{|V|}\ll |P|^{|V|+1}$.
    Hence, we now suppose that the points of $P$ are not colinear through the origin and $P$ does not contain $(0,0,0)$.

    For $t \geq 2$,
    we now define $\widetilde{\mathcal{R}}_{\geq t}(P)$ to be the subset of $\mathcal{R}_{\geq t}(P)$ of isometries $\theta$ where the points in $P \cap \theta P$ are not contained in a line through the origin.
    Let $Q$ be the projection of the points of $P$ onto $\mathbb{S}^2$.

    \begin{claim}
        $\sum_{t= 2}^{|P|}t^{|V|-1}\big|\mathcal{R}_{\geq t}(P)\big| \ll |P|^{|V|+1} + \sum_{t= 2}^{|P|}t^{|V|-1}\big|\widetilde{\mathcal{R}}_{\geq t}(P)\big|$.
    \end{claim}

    \begin{proof}
        Let $S \subset \Gamma$ be the set of isometries for which the points in $P \cap \theta P$ are contained in a line through the origin.
        Using the quotient $\sim_P$,
        we observe that for each $\theta \in S$,
        there exists an ordered pair $(x,y) \in Q^2$ such that $\theta(x) = y$.
        It hence follows that the equivalence classes $S/\sim_P$ can be considered as ordered pairs of points in $Q^2$.
        Choose representatives $\theta_1,\ldots,\theta_{n}$ for the elements of $\mathcal{R}_{=1}(Q)$.
        We now suppose that we order these representatives such that only $\theta_1,\ldots,\theta_m$ map a point in $P$ to itself (which implies $m \leq |P|$). We make the definitions
        \begin{equation*}
            A_t := \{ \theta_i : \theta_i \in \mathcal{R}_{\geq t}(P), ~ i \leq m\}, \qquad B_t := \{ \theta_i : \theta_i \in \mathcal{R}_{\geq t}(P), ~ i > m\}.
        \end{equation*}
        Since $m \leq |P|$,
        we have $|A_t| \leq |P|$ for each $t \geq 2$.
        Now we turn to $B_t$.
        For each $\theta_i \in B_t$,
        let $x_i,y_i \in Q$ be the points such that $\theta(x_i) = y_i$.
        If $\theta_i,\theta_j \in B_t$ with $x_i=x_j$ and $y_i = y_j$, then $i=j$ also.
        Hence,
        as each $\theta_i$ provides a distinct ordered pair $(x_i, y_i) \in Q$, with $x_i$ and $y_i$ being the projection of $t$ points each, $B_t$ contains at most $(|P|/t)^2 = |P|^2/t^2$ points (in fact, there must be strictly less since this count allows for $x_i = y_i$).
        We now see that
        \begin{align*}
            \sum_{t= 2}^{|P|}t^{|V|-1}\big|\mathcal{R}_{\geq t}(P)\big| &\leq  \sum_{t= 2}^{|P|} t^{|V|-1} |A_t| + \sum_{t= 2}^{|P|} t^{|V|-1} |B_t| + \sum_{t= 2}^{|P|}t^{|V|-1}\big|\widetilde{\mathcal{R}}_{\geq t}(P)\big| \\
            &\ll  |P|\sum_{t= 2}^{|P|} t^{|V|-1} + |P|^2 \sum_{t= 2}^{|P|} t^{|V|-3}+ \sum_{t=2}^{|P|}t^{|V|-1}\big|\widetilde{\mathcal{R}}_{\geq t}(P)\big| \\
            &\ll  |P|^{|V|+1} + |P|^{|V|}+ \sum_{t=2}^{|P|}t^{|V|-1}\big|\widetilde{\mathcal{R}}_{\geq t}(P)\big| \\
            &\ll  |P|^{|V|+1} + \sum_{t=2}^{|P|}t^{|V|-1}\big|\widetilde{\mathcal{R}}_{\geq t}(P)\big| . \qedhere
        \end{align*}
    \end{proof}

    \begin{claim}
        $\sum_{t= 2}^{|P|}t^{|V|-1}\big|\widetilde{\mathcal{R}}_{\geq t}(P)\big| \ll \sum_{t= 2}^{|Q|}(t+|P|-|Q|)^{|V|-1}\big|\mathcal{R}_{\geq t}(Q)\big|$.
    \end{claim}

    \begin{proof}
        First, set $P' = P \cup Q$.
        Then $|\widetilde{\mathcal{R}}_{\geq t}(P)| \leq |\widetilde{\mathcal{R}}_{\geq t}(P')|$ for each $t \geq 2$.
        It hence is sufficient to prove the result for the case where $P=P'$ (i.e., $Q \subset P$).

        Choose any $\theta \in \widetilde{\mathcal{R}}_{\geq t}(P)$.
        Then $\theta \in \mathcal{R}_{= s}(Q)$ for some $s \geq t - (|P|-|Q|)$.
        Hence, $|\widetilde{\mathcal{R}}_{\geq t}(P)| \leq |\mathcal{R}_{\geq t - |P|+|Q|}(Q)|$,
        with $|\mathcal{R}_{\geq t - |P|+|Q|}(Q)| = |\mathcal{R}_{\geq 2}(Q)|$ if $t \leq |P| - |Q| + 2$.
        This then gives the substitution
        \begin{equation*}
            \sum_{t= 2}^{|P|}t^{|V|-1}\big|\widetilde{\mathcal{R}}_{\geq t}(P)\big| \leq \sum_{t= 2}^{|P|}t^{|V|-1}\big|\mathcal{R}_{\geq t - |P|+|Q|}(Q)\big| \ll \sum_{t= 2}^{|Q|}(t+|P|-|Q|)^{|V|-1}\big|\mathcal{R}_{\geq t}(Q)\big|. \qedhere
        \end{equation*}
    \end{proof}

    By combining all of the above with \Cref{lem:hardpartg},
    we obtain the following:
    \begin{align*}
        |\mathcal{E}_{V,\Gamma}(P)| &\ll |P|^{|V|+1} + \sum_{t= 2}^{|Q|}(t+|P|-|Q|)^{|V|-1}\big|\mathcal{R}_{\geq t}(Q)\big|  \\
        &\ll |P|^{|V|+1} +  |Q|^3\sum_{t= 2}^{|Q|}(t+|P|-|Q|)^{|V|-3}\\
        &\ll |P|^{|V|+1} +  |Q|^3 |P|^{|V|-2} \\
        &\ll |P|^{|V|+1}.
    \end{align*}
    This now concludes the proof.
\end{proof}

We can then obtain the following result by applying an analogous method to that given in the proof of \Cref{lem:rotations}.

\begin{lemma}\label{lem:energyrotationonly}
    Let $V$ be a fixed finite set with $|V| \geq 4$ and let $\Gamma$ be a subgroup of $\Othree$.
    Then for any $P \subset \mathbb{R}^3$,
    we have
    \begin{equation*}
        |\mathcal{E}_{V,\Gamma}(P)| \ll |P|^{|V|+1}.
    \end{equation*}
\end{lemma}

Another application for rich sets is the group of pseudo-Euclidean isometries, $\OPseudo$.

\begin{theorem}[Rudnev and Roche-Netwon \cite{rnr15}]\label{thm:psenergy}
    Fix $\Gamma = \mathbb{R}^2 \rtimes \SO^+(1,1)$.
    Then there exists a constant $C>0$ so that the following holds:
    for any finite point set $P\subset \mathbb{R}^2$ not contained in the translate of one of the lines $\{(t,\pm t) : t \in \mathbb{R}\}$, and any $2 \leq t \leq |P|$, the size of $\mathcal{R}_{\geq t}(P)$ is bounded as follows:
    \[\big|\mathcal{R}_{\geq t}(P)\big| \leq C \frac{|P|^3}{t^2}.\]
\end{theorem}

We obtain the following result from \Cref{thm:psenergy}.

\begin{lemma}\label{lem:pseudoeuclideanenergy}
    Let $V$ be a fixed finite set with $|V| \geq 3$ and let $\Gamma$ be a subgroup of $\mathbb{R}^2 \rtimes \OPseudo$.
    Then for any $P \subset \mathbb{R}^2$,
    we have
    \begin{equation*}
        |\mathcal{E}_{V,\Gamma}(P)| \ll |P|^{|V|+1}.
    \end{equation*}
\end{lemma}

\begin{proof}
    The proof is almost identical to \Cref{lem:rotations},
    except that we replace $\theta_x$ with the matrices
    \begin{align*}
        \theta_{1} := 
        \begin{pmatrix}
            1 & 0 \\
            0 & -1
        \end{pmatrix}, \qquad
        \theta_{2} := 
        \begin{pmatrix}
            -1 & 0 \\
            0 & 1
        \end{pmatrix}, \qquad
        \theta_{3} := 
        \begin{pmatrix}
            -1 & 0 \\
            0 & -1
        \end{pmatrix},
    \end{align*}
    and we replace $P'$ with
    \begin{equation*}
        P \cup \theta_{1} P \cup \theta_{2} P \cup \theta_{3} P. \qedhere
    \end{equation*}
\end{proof}

Our final application is for affine groups formed from translations combined with finitely many isometries, such as for the $\ell_p$ metric when $p$ is even and $p \geq 4$.

\begin{lemma}\label{lem:energytranslationonly}
    Let $V$ be a fixed finite set and let $\Gamma$ be chosen such that, for some finite set of $d \times d$ matrices $\Gamma_\ell$ and some linear space $\Gamma_t \leq \mathbb{R}^d$ of translations, we have $\Gamma = \Gamma_t \rtimes \Gamma_\ell$.
    Then for any $P \subset \mathbb{R}^d$,
    we have
    \begin{equation*}
        |\mathcal{E}_{V,\Gamma}(P)| \ll |P|^{|V|+1}.
    \end{equation*}
\end{lemma}

\begin{proof}
    The result is trivially true if $|V|=1$, so we may suppose that $|V| \geq 2$.
    Since $|\mathcal{R}_{\geq t}(P)| \leq |\Gamma_\ell||P|$ for each $t \geq 1$,
    it follows from \Cref{lem:hardpartg} that for some arbitrary $u \in V$,
    \begin{equation*}
        |\mathcal{E}_{V,\Gamma}(P)| \ll |\Gamma_{\ell}||P|\sum_{t=1}^{|P|}t^{|V|-1} \ll |\mathcal{E}_{V\setminus \{u\},\Gamma}(P)| + |P|^{|V|+1}.
    \end{equation*}
    The result now follows from induction on the size of $V$.
\end{proof}

\subsection{Proof of \texorpdfstring{\Cref{thm:Hypergraph2D}}{second main theorem}}

In this subsection we will prove \Cref{thm:Hypergraph2D}.

\Hypergraph*

We will prove the above for a spanning tree of $G$. We break our point set into subsets $P_0,P_1,\ldots, P_{n-1}$ we then count the number of edges that can go between $P_i$ and $P_j$ when this an edge on the spanning tree. As $g$ is a general (anti-)symmetric map we first deal with points for which all points on $C$ are $g$-`equidistant'. 

For the remainder of the section, we fix the following notation.
Let $C$ be an irreducible algebraic curve with degree at most $D$ and let $g:(\mathbb{R}^2)^2 \rightarrow \mathbb{R}$ be a polynomial (symmetric or anti-symmetric) where $g|_{C^2}$ is not constant.
With this,
we define the set
\begin{equation*}
    C' = \left\{y \in C: x \mapsto g(x,y) \text{ is constant for $x \in C$} \right\}.
\end{equation*}
We show that $C'$ is a set controlled by the degree of $C$ and $g$.
\begin{lemma}\label{lem:SmallNumberOfBadPoints}
    The set $C'$ contains at most $D \cdot \deg (g)$ points.
\end{lemma}
The proof relies on the fact that $C'$ is a subset of the intersection (in $\R^2$) of $C$ and a level set of $g$. As $g|_{C^2}$ is not constant and $C$ is irreducible we can use B\'ezout's theorem to conclude that $C'$ contains at most $D\cdot\deg(g)$ points.

\begin{proof}
It follows from the (anti-)symmetry of $g$ and the irreducibility of $C$ that if $y \in C'$,
then the map
\begin{equation*}
    x \mapsto g(y, x)
\end{equation*}
is also constant.
Hence, for any $y_1,y_2 \in C'$ and any $x \in C$,
we have
\begin{equation*}
    g(x, y_1) = g(y_1,x) =g(y_1,y_2) = g(x, y_2).
\end{equation*}
We now fix the constant $\alpha \in \mathbb{R}$ so that $g(x,y) = \alpha$ for all $y \in C'$.

    By B\'{e}zout's theorem,
    either $C\cap Z(g-\alpha)$ contains at most $D \cdot \deg (g)$ points or $C$ and $g-\alpha$ share a common factor.
    Since $C$ is irreducible,
    this then implies $C \cap Z(g-\alpha) = C$, which is a contradiction as we assumed that $g|_{C^2}$ is non-constant.
\end{proof}

\begin{proof}[Proof of \Cref{thm:Hypergraph2D}]
   Assume, without loss of generality, that $G$ is a connected tree with exactly $|V|-1$ edges. We count distinct $g$-realisations of $G$ by breaking our point set up into $|V|$ many pieces and arguing that we have at least
   \begin{equation*}
       \left(\frac{1}{D\cdot\deg(g)} \left\lfloor\frac{|P|-D\cdot\deg(g)}{|V| } \right\rfloor\right)
   \end{equation*}
   choices for each edge in $G$.
 
    Fix $P' = P\setminus C'$.
    By \Cref{lem:SmallNumberOfBadPoints},
    the set $P'$ is non-empty and contains at least $|P| - D \cdot \deg (g)$ points.
    Partition $P'$ into sets $\{P_v : v \in V\}$ so that $\lfloor |P'|/|V| \rfloor \leq |P_v| \leq \lceil |P'|/|V| \rceil$ for each $v \in V$.
    
    \begin{claim}\label{claim:edgelengths}
        For any edge $\{v,w\} \in G$ and for any $x_0 \in P_{v}$,
        the number of distinct elements of the set $S = \{g(x_0,y) : y \in P_{w}\}$ is at least 
        \begin{equation*}
            \frac{1}{D\cdot\deg(g)} \left\lfloor\frac{|P|-D\cdot\deg(g)}{|V|} \right\rfloor.
        \end{equation*}
    \end{claim}
    
    \begin{proof}[Proof of Claim \ref{claim:edgelengths}]
        Since $x_0 \notin C'$,
        the argument above tells us that the size of any level set $\{y \in C : g(x_0,y) = t \}$ is at most $D\cdot\deg(g)$.
        Hence, the set $\{g(x_0,y) : y \in P_{w}\}$ contains at least $|P_w|/(D\cdot\deg(g))$ distinct elements.
        The result now follows as
        \begin{equation*}
            |P_w| \geq \left\lceil \frac{|P'|}{|V|}\right\rceil \geq \left\lfloor \frac{|P|-D\cdot\deg(g)}{|V|}\right\rfloor.\qedhere
        \end{equation*}
    \end{proof}

    Fix an order $v_0, \ldots, v_{n-1}$ so that for all $2\leq i < n$ the vertex $v_i$ is connected to some $v_j$ with $j<i$. Let $G_i$ be the induced subgraph of $G$ on the vertices $v_0,\ldots, v_i$.
    Next, fix the following induction hypothesis for each $2 \leq i < n$:
    \begin{equation}\label{eq:induct}
        \left|f_{g,G_i}\left( \prod_{j=0}^i P_{v_j} \right) \right| \geq \left( \frac{1}{D\cdot\deg(g)} \left\lfloor\frac{|P|-D\cdot\deg(g)}{|V| } \right\rfloor \right)^{i-1}.
    \end{equation}
    The base case of $i=2$ is exactly the statement of \Cref{claim:edgelengths}. Suppose \cref{eq:induct} holds for $i=m < n-1$.
    It follows from \Cref{claim:edgelengths} that
    \begin{equation*}
        \left|f_{g,G_{m+1}}\left( \prod_{j=0}^{m+1} P_{v_j} \right) \right| \geq  \left( \frac{1}{D\cdot\deg(g)} \left\lfloor\frac{|P|-D\cdot\deg(g)}{|V| } \right\rfloor \right)\left|f_{g,G_{m}}\left( \prod_{j=0}^{m} P_{v_j} \right) \right| .
    \end{equation*}
    Hence, \cref{eq:induct} holds for $i=m+1$.
    Using \cref{eq:induct} with $i=n-1$, we see that
    \begin{equation*}
       \Big| f_{g,G}\left(P^V\right) \Big| \geq \left|f_{g,G_{n-1}}\left( \prod_{j=0}^{n-1} P_{v_j} \right) \right| \geq \left( \frac{1}{D\cdot\deg(g)} \left\lfloor\frac{|P|-D\cdot\deg(g)}{|V| } \right\rfloor \right)^{|V|-1}.
    \end{equation*}
    This concludes the proof.
\end{proof}

\subsection{Necessity of low degree curve condition for \texorpdfstring{\Cref{thm:energyg}}{key theorem}}

It is tempting to believe that the low degree curve requirement of \Cref{thm:energyg} is not required, especially with the application of \Cref{thm:Hypergraph2D}.
Unfortunately, \Cref{thm:Hypergraph2D} is less helpful when the expected lower bound is of the magnitude of $|P|^{|V|}$ (for example, if there are finitely many $g$-isometries), even if we restrict to curves where the polynomial $g$ is not constant. To illustrate this point,
take for example the polynomial
\begin{align*}
    g: \mathbb{R}^2 \times \mathbb{R}^2 \rightarrow \mathbb{R}, ~ (x,y) = \big( (x_1,x_2), (y_1,y_2) \big)  \mapsto x \cdot y + (\|x\|^2 \|y\|^2 - 1) x_1^3 y_1^3 .
\end{align*}
Note that the polynomial $g$ is identical to the dot product if we restrict to points in $\mathbb S^1 \times \mathbb S^1$. However, since the Jacobian of $g$ is given by
\begin{align*}
    \frac{d}{d p} g(p,q) = 
    \begin{pmatrix}
        y_1 + 2(\|x\|^2 \|y\|^2 - 1) x_1^2 y_1^3 + 2 \|y\|^2 x_1^4 y_1^3 \\
        y_2 + 2\|y\|^2 x_1^3 x_2 y_1^3
    \end{pmatrix}.
\end{align*}
we see that $g$ is not constant even if we restrict its domain to points in $\mathbb{S}^1$.

If we take $(K_5,p)$ to be the framework where
\begin{equation*}
    p(0) = (8, -5), \quad
    p(1) = (0, -6), \quad
    p(2) = (5, 10), \quad
    p(3) = (-3, -7), \quad
    p(4) = (-4, -4),
\end{equation*}
then we compute that $\rank \mathrm{J} f_{g,G}(p) = 10$.
Since $(K_5,p)$ is affinely spanning, this informs us of two things:
(i) $\dim \Gamma_g = 0$, and thus $\Gamma_g$ is finite, and (ii) $(K_5,p)$ is infinitesimally $g$-rigid, and hence $K_5$ is $g$-rigid.

As a consequence of \Cref{cor:finiteisom}, for any point set $P \subset \mathbb{R}^2$ where at most $\frac{1}{100}|P|$ points are contained in any given curve of degree at most $4 \cdot 5=20$,
we have 
\begin{equation*}
    \Big| f_{g,K_5}(P_n^5) \Big| = \Theta (|P|^5)
\end{equation*}
However,
if we choose the point set
\begin{equation*}
    P_n := \Big\{ \big(\cos (k/2\pi) , \sin (k/2\pi) \big) : k \in \{0,1,\ldots,n-1\}  \Big\},
\end{equation*}
then, since $g$ is identical to the dot product when restricted to the circle $\mathbb{S}^1$, we have that
\begin{equation*}
    \Big| f_{g,K_5}(P_n^5) \Big| \ll |P_n|^{4}.
\end{equation*}

\subsection{Constructing tight sets}

In \Cref{sec:SpecificExamplesOfInterest},
we explained why each of our constructed lower bounds was indeed tight.
In fact, when the $g$-isometry group $\Gamma_g$ has positive dimension, the lower bound on $f_{g,G}\left(P^V\right)$ is always $O(|P|^{|V|-1})$,
even if we are required to avoid finitely many algebraic hypersurfaces (such as avoiding diagonal lines in the case where $g$ is the skew-symmetric bilinear form given in \cref{eq:skewsymm}).

\begin{prop}\label{prop:tightconstruction}
    If $\dim \Gamma_g \geq 1$,
    then for any proper algebraic set $C \subseteq \mathbb R^d$,
    there exists a sequence $(P_n)_{n \in \mathbb{N}}$ where $P_n$ is a $n$-point subset of $\mathbb{R}^d \setminus C$ so that for any $k$-uniform hypergraph $G=(V,E)$,
    we have
    \begin{equation*}
        f_{g,G}(P_n^V) \leq |V||P_n|^{|V|-1}.
    \end{equation*}
\end{prop}

\begin{proof}
    Choose $\lambda \in \mathfrak{g}_g$.
    Then the set $\Gamma_0 := \{ \exp(t \lambda) : t \in \mathbb{R}\}$ is a 1-dimensional Lie subgroup of $\Gamma_g$,
    and hence either $\Gamma_0$ is compact and isomorphic (as a Lie group) to the circle group,
    or $\Gamma_0$ is not compact and isomorphic (as a Lie group) to $(\mathbb{R},+)$.
    In either case,
    there exists $\theta \in \Gamma_0$ such that the subgroup $\langle \theta \rangle$ is isomorphic (just as a group) to $\mathbb{Z}$.
    For each positive integer $n$,
    define the set
    \begin{equation*}
        S_n := \left\{ x \in \mathbb{R}^d : \theta^n (x) = x  \text{ or } \theta^n (x) \in C  \right\}.
    \end{equation*}
    Each set $S_n$ is contained in a proper algebraic hypersurface of $\mathbb{R}^d$,
    and so their union is nowhere dense in $\mathbb{R}^d$.
    Choose $x \in \mathbb{R}^d \setminus \left( C \cup \bigcup_{n \in \mathbb{N}} S_n \right)$.
    Then $\theta^n (x) \notin C$ for each $n \geq 0$, and $\theta^i (x) \neq \theta^j(x)$ for $0 \leq i < j$,
    as otherwise $\theta^{j-i}(x) =x$, contradicting that $x \notin S_{j-i}$.
    For each $n$,
    we define the $n$-point set
    \begin{equation*}
        P_n := \left\{ \theta^n (x) : n \in \{0, 1, \ldots, n-1\} \right\}.
    \end{equation*}
    For any $p \in P_n^V$,
    we can always apply an isometry $\theta^{-j}$ for some $j \in \{0,\ldots,n-1\}$ to obtain a realisation $q \in P_n^V$ with one vertex positioned at $x$.
    Hence,   
    \begin{equation*}
        f_{g,G}(P_n^V) \leq \Big| \left\{g(p) \in P_n^V : p(v) = x \text{ for some } v \in V \right\}  \Big| \leq |V||P|^{|V|-1}.\qedhere
    \end{equation*}
\end{proof}

Even in the case where there exist uncountably many `bad' sets we wish to avoid,
we can adapt \Cref{prop:tightconstruction} easily to approach this.
Take for example the case where $g$ is the pseudo-Euclidean metric given in \cref{eq:pseudo}.
Here, we wish to avoid point sets contained within a single line parallel to one of the diagonals $\{ (t,\pm t): t \in \mathbb{R}\}$, as in such cases we have $\big|f_{g,G}(P^V)\big| = \{0\}$.
If we assume that each set $P_n$ contains the point $(0,0)$,
we then proceed with the construction given in \Cref{prop:tightconstruction} with $C$ being the union of the two diagonals.

\section{Proof of \texorpdfstring{\Cref{prop:fewCongClassesg}}{key proposition}} \label{sec:prop2.8proof}

In this section, we prove the following result from \Cref{sec:graphicalgrigid}.

\fewCongClassesg*

Our first step is to improve \cite[Proposition 4.8]{cruik}.
With this aim in mind, we start with the following lemma.

\begin{lemma}\label{prop:asimowroth2}
    Let $G$ be a $k$-uniform hypergraph and let $(G,p)$ be a framework in $\mathbb{R}^d$.
    If the set of vectors $\{p(v) : v \in V\}$ affinely spans $\mathbb{R}^d$,
    then the following properties are equivalent:
    \setlength{\parskip}{0pt}
    \begin{enumerate}[label=(\roman*)]
        \item $(G,p)$ is locally $g$-rigid.
        \item The connected component of $f^{-1}_{g,G}(f_{g,G}(p))$ containing $p$ is contained within $\Gamma_g \cdot p$.
    \end{enumerate}
\end{lemma}

\begin{proof}
    For any realisation $q$ of $G$,
    we define $\mathcal{O}_q := \Gamma_g \cdot q$, that is, $\mathcal{O}_q$ is the orbit of $q$ under $\Gamma_g$.
    When $q$ is affinely spanning,
    the map from $\Gamma_g$ to $\mathcal{O}_q$ given by $\theta \mapsto \theta \cdot q$ is an injective map (in fact, it is the restriction of a linear map),
    and hence $\mathcal{O}_q$ is a smooth manifold.
    We further define $C_p$ to be the connected component of $f^{-1}_{g,G}(f_{g,G}(p))$ containing $p$.
    As algebraic sets have finitely many connected components (see \cite[Theorem 2.4.4]{boch}),
    the set $C_p$ is a clopen subset of $f^{-1}_{g,G}(f_{g,G}(p))$.
    
    We first prove that $(i) \implies (ii)$. Suppose that $(G,p)$ is locally $g$-rigid.
    By definition, there exists an open neighbourhood $U \subset (\mathbb{R}^d)^V$ of $p$ such that
    \begin{equation}\label{eq1:asimowroth2}
        f^{-1}_{g,G}(f_{g,G}(p)) \cap U= \mathcal{O}_p \cap U.
    \end{equation}
    We note here that for each $\theta \in \Gamma_g$ and $q = \theta \cdot p$,
    we can see from \cref{eq1:asimowroth2} that
    \begin{equation}\label{eq2:asimowroth2}
        f^{-1}_{g,G}(f_{g,G}(q)) \cap (\theta \cdot U) = \mathcal{O}_p \cap (\theta \cdot U).
    \end{equation}
    Using \cref{eq2:asimowroth2},
    we see that $\mathcal{O}_p \cap C_p$ is an open subset of $C_p$.
    Since $\mathcal{O}_p$ is a smooth manifold,
    the set $\mathcal{O}_p \cap C_p$ is also closed.
    Hence, $\mathcal{O}_p \cap C_p$ is a clopen subset of $C_p$.
    As $C_p$ is a connected topological space,
    we have $\mathcal{O}_p \cap C_p = C_p$ as required.

    We now prove that $(ii)\implies(i)$. Suppose that $\mathcal{O}_p \cap C_p = C_p$.
    Since $C_p$ is a clopen subset of $f^{-1}_{g,G}(f_{g,G}(p))$,
    there exists an open set $U \subset (\mathbb{R}^d)^V$ so that
    \begin{equation*}
        f^{-1}_{g,G}(f_{g,G}(p)) \cap U = C_p =\mathcal{O}_p \cap C_p,
    \end{equation*}
    hence $(G,p)$ is locally $g$-rigid.
\end{proof}

\begin{remark}
When $g$ is the Euclidean distance metric, the second condition in \Cref{prop:asimowroth2} is called \emph{continuous rigidity}.
This distinction is irrelevant in this case since local and continuous rigidity are equivalent,
however this equivalence is not true for other metrics such as normed spaces; see \cite[Example 4.19]{dewar21} for an example where local rigidity is a strictly stronger property than continuous rigidity.
\end{remark}

\begin{lemma}\label{lem:genericrealisationnumber}
    Let $G$ be a $g$-rigid $k$-uniform hypergraph with at least $d+1$ vertices.
    If $(G,p)$ is a generic framework,
    then every framework $(G,q)$ with $f_{g,G}(q) = f_{g,G}(p)$ is affinely spanning.
\end{lemma}

\begin{proof}
    Fix $X \subset (\mathbb{R}^d)^V$ to be the algebraic set of realisations that are not affinely spanning,
    and $X',Y \subset \mathbb{R}^E$ to be the Zariski closures of the sets $f_{g,G}(X)$ and $f_{g,G}((\mathbb{R}^d)^V)$ respectively.
    We note here that $Y$ is irreducible,
    since its preimage is a linear space,
    and $\dim Y = d|V| - \dim \Gamma_g$ since $G$ is $g$-rigid.

    \begin{claim}
        The algebraic set $X'$ is a proper algebraic subset of $Y$.
    \end{claim}

    \begin{proof}
        If we set
        \begin{equation*}
            U := \left\{ q \in X : \rank {\rm J} f_{g,G}(q) \geq \rank {\rm J} f_{g,G}(q') \text{ for all } q' \in X \right\},
        \end{equation*}
        then $U$ is a non-empty Zariski open subset of $X$,
        $U' := f_{g,G}(U)$ is an non-empty Zariski open subset of $X'$,
        and the Zariski closure of $U'$ is exactly the set $X'$.
        It hence suffices to show that $U'$ has dimension strictly less than $d|V|-\dim \Gamma_g$.

        Choose a point $q \in U$.
        If $\rank {\rm J} f_{g,G}(q) < d|V|-\dim \Gamma_g$,
        then 
        \begin{equation*}
            \dim U' = \rank {\rm J} f_{g,G}(q) < d|V|-\dim \Gamma_g
        \end{equation*}
        and we are done.
        Suppose instead that $r := \rank {\rm J} f_{g,G}(q) = d|V|-\dim \Gamma_g$.
        By the constant rank theorem,
        there exists an open neighbourhood $A_q \subset (\mathbb{R}^d)^V$ of $q$ such that $B_q = f_{g,G}(A_q)$ is a smooth manifold which is diffeomorphic to the $r$-dimensional open ball, and there exist smooth diffeomorphisms $\alpha : A_q \rightarrow \mathbb{R}^{d|V|}$ and $\beta : B_q \rightarrow \mathbb{R}^r$ such that for each $x = (x_i)_{i=1}^{d|V|} \in \alpha(A_q)$ we have
        \begin{equation*}
            \beta \circ f_{g,G} \circ \alpha^{-1}(x) = (x_i)_{i=1}^{r};
        \end{equation*}
        i.e., the map $f :=  \beta \circ f_{g,G} \circ \alpha^{-1}$ is a projection of points to their first $r$ coordinates.
        We can retro-actively alter $A_q$ (without changing that $B_q$ is diffeomorphic to the $r$-dimensional open ball) so that the set $\alpha(A_q)$ is an open convex set,
        and hence the intersection of $\alpha(A_q)$ with any affine space is connected.
        With this set-up, it is simple to see that the fibres of $f$ (and hence $f_{g,G}|_{A_q}^{B_q}$) are connected.
        
        Now choose $\tilde{q} \in A_q \setminus X$.
        Here, we have that $\rank {\rm J} f_{g,G}(\tilde{q}) = \rank {\rm J} f (\alpha(\tilde{q})) = r$.
        Hence, $(G,\tilde{q})$ is infinitesimally $g$-rigid.
        By \Cref{prop:asimowroth1},
        $(G,\tilde{q})$ is locally $g$-rigid.
        Recalling that $\tilde{q}$ is affinely spanning, it follows from \Cref{prop:asimowroth2} that the connected component of the fibre $f^{-1}_{g,G}(f_{g,G}(\tilde{q}))$ that contains $\tilde{q}$ is contained in $\Gamma_g \cdot \tilde{q}$.
        Since $f^{-1}_{g,G}(f_{g,G}(\tilde{q})) \cap A_q$ is connected (by our retroactive choice of $A_q$),
        it follows that 
        \begin{equation*}
            f^{-1}_{g,G}(f_{g,G}(\tilde{q})) \cap A_q \subset \Gamma_g \cdot \tilde{q}.
        \end{equation*}
        As $\Gamma_g$ acts on $(\mathbb{R}^d)^V \setminus X$ (i.e., affinely spanning realisations can only be mapped to affinely spanning realisations by non-singular affine transformations),
        we have
        \begin{equation*}
            \left(f^{-1}_{g,G}(f_{g,G}(\tilde{q})) \cap A_q \right) \cap X \subseteq \left(\Gamma_g \cdot \tilde{q} \right) \cap X = \emptyset \quad \implies \quad f_{g,G}(\tilde{q}) \notin f_{g,G}(X \cap A_q); 
        \end{equation*}
        furthermore, this holds for any $\tilde{q} \in A_q \setminus X$.
        As $q \in U$ was chosen arbitrarily,
        it follows that $U'$ cannot be a Zariski open subset of $Y$.
        This now concludes the proof of the claim.
    \end{proof}

    As $p$ is generic,
    its image is also generic in $Y$. That is, the only polynomial equations with coefficients in the field $\mathbb{Q}(g)$ that are satisfied by $f_{g,G}(p)$ are the equations defining $Y$.
    Since $X'$ is defined by polynomial equations with coefficients in the field $\mathbb{Q}(g)$,
    it follows that $f_{g,G}(p)$ is not contained in $X'$.
    Hence the set $f_{g,G}^{-1}(f_{g,G}(p))$ only contains affinely spanning realisations.
\end{proof}

We also require the following technical results regarding Lie group actions and real algebraic geometry.

\begin{lemma}\label{lem:propermap}
    For $|V| \geq d+1$,
    let $Z \subset (\mathbb{R}^d)^V$ be the affinely spanning realisations and let $\Gamma$ be a closed subgroup of $\Aff (d)$.
    Then the map
    \begin{equation*}
        \Phi : \Gamma \times Z \rightarrow Z \times Z, ~ (\theta,p) \mapsto (\theta \cdot p, p)
    \end{equation*}
    is well-defined, injective and proper.
\end{lemma}

\begin{proof}
    Using the superscript notation $(p^n)_{n \in \mathbb{N}}$ to denote a sequence of realisations in $Z$, take any sequence $((\theta_n,p^n))_{n \in \mathbb{N}}$ in $\Gamma \times Z$ where $\Phi(\theta_n,p^n) \rightarrow (q , p)$ as $n \rightarrow \infty$ for some $q \in Z$.
    It is immediate that $p = \lim_{n \rightarrow \infty} p^n \in Z$.
    It is now sufficient to show that $(\theta_n)_{n \in \mathbb{N}}$ is convergent.
        
    For each $\theta_n$, set $\lambda_n, x_n$ to be the unique linear transform and vector pair where $\theta_n(x) = \lambda_n (x) + x_n$ for each $x \in \mathbb{R}^d$.
    As $p$ is affinely spanning,
    there exists vertices $v_0,\ldots,v_{d}$ so that $p(v_0),\ldots, p(v_d)$ are affinely independent.
    Using this, we define the following $(d+1) \times (d+1)$ matrices for each $n \in \mathbb{N}$:
    \begin{align*}
        \Theta_n :=
        \begin{pmatrix}
            \lambda_n & x_n \\
            \mathbf{0}^T & 1
        \end{pmatrix},
        \qquad
        P_n :=
        \begin{pmatrix}
            p^n(v_0) & \cdots & p^n(v_d) \\
            1 & \cdots & 1
        \end{pmatrix},\\
        P :=
        \begin{pmatrix}
            p(v_0) & \cdots & p(v_d) \\
            1 & \cdots & 1
        \end{pmatrix},
        \qquad
        Q :=
        \begin{pmatrix}
            q(v_0) & \cdots & q(v_d) \\
            1 & \cdots & 1
        \end{pmatrix}.
    \end{align*}
    Importantly, the matrix $P$ and the matrices $P_n$ for sufficiently large $n$ are non-singular.
    Our previous conditions now imply $\Theta_n P_n \rightarrow Q$ and $P_n \rightarrow P$ as $n \rightarrow \infty$,
    with the second limit implying that $P_n^{-1} \rightarrow P^{-1}$ as $n \rightarrow \infty$ (assuming $n$ is sufficiently large enough that $P_n$ is invertible).
    Hence,
    \begin{equation*}
        \Theta_n = (\Theta_n P_n) (P_n^{-1}) \rightarrow Q P^{-1},
    \end{equation*}
    which in turn implies $(\theta_n)_{n \in \mathbb{N}}$ is convergent.
\end{proof}

\begin{theorem}[Milner-Thom theorem \cite{milnor,thom}]\label{thm:milnerthom}
    Let $f=(f_1,\ldots,f_n):\mathbb{R}^m \rightarrow \mathbb{R}^n$ be a polynomial map where each map $f_i$ has degree at most $d$.
    Then for any point $\lambda \in \mathbb{R}^n$,
    the fibre $f^{-1}(\lambda)$ contains at most $d(2d-1)^{n-1}$ real connected components.
\end{theorem}

We are now ready to prove the following strengthening of \cite[Proposition 4.8]{cruik}.

\begin{prop}\label{prop:genericrealisationnumber}
    Let $G$ be a $g$-rigid $k$-uniform hypergraph with at least $d+1$ vertices.
    Then there exists a positive integer $c_G$ such that for each generic framework $(G,p)$,
    the set $f_{g,G}^{-1}(f_{g,G}(p))/ \Gamma_g$ contains at most $c_G$ elements.
\end{prop}

\begin{proof}
    Fix $\Phi$ to be the map given in \Cref{lem:propermap}.    
    As $\Phi$ is proper,
    it follows from a standard result regarding proper Lie group actions on smooth manifolds (see for example \cite[Corollary 4.1.22]{abrahammarsden}) that for any $q \in Z$, the set $\Gamma_g \cdot q$ is an embedded smooth manifold of dimension $\dim \Gamma_g$ given by the smooth embedding $\theta \mapsto \theta \cdot q$.

    Now choose generic $p \in (\mathbb{R}^d)^V$.
    By \Cref{prop:adaptedgeneric},
    the set $f_{g,G}^{-1}(f_{g,G}(p))$ is a smooth algebraic set.
    Moreover, by \Cref{thm:milnerthom}, there exists a constant value $t$, independent of our choice of $p$, so that $f_{g,G}^{-1}(f_{g,G}(p))$ has at most $t$ connected components. 
    By \Cref{lem:genericrealisationnumber},
    every realisation $q \in f_{g,G}^{-1}(f_{g,G}(p))$ is affinely spanning.
    Thus for any $q \in f_{g,G}^{-1}(f_{g,G}(p))$,
    the smooth manifold $\Gamma \cdot q$ consists of exactly $k$ connected components of $f_{g,G}^{-1}(f_{g,G}(p))$,
    where $k$ is the number of connected components of $\Gamma$.
    Because of this, the quotient space $f_{g,G}^{-1}(f_{g,G}(p))/ \Gamma_g$ consists of at most $t/k$ points.
    We now conclude the proof by setting $c_G = \lceil t/k \rceil$.
\end{proof}

We are now ready to prove \Cref{prop:fewCongClassesg}.

\begin{proof}[Proof of \Cref{prop:fewCongClassesg}]
    Fix $\Phi$ to be the proper map described in \Cref{lem:propermap}.
    Using a standard result regarding proper Lie group actions on smooth manifolds (see for example \cite[Theorem 5.1.16]{manifold}),
    we have that the space $X := Z / \Gamma_g$ is a smooth manifold,
    and the map $f : X \rightarrow \mathbb{R}^E$ which maps every equivalence class $\tilde{q}$ to $f_{g,G}(q)$ (with $q \in \tilde{q}$ chosen arbitrarily) is a well-defined smooth map.
    We will now opt to abuse notation for the remainder of the proof by denoting each equivalence class $\tilde{q}$ of $X$ by a representative point $q$ contained within.
    
    We now split the proof into two cases.    
    First, assume that $(G,p)$ is minimally infinitesimally $g$-rigid.    
    By applying the constant rank theorem to each $q \in \mathcal{S}$,
    we obtain open sets $A_q \subset X$ and $B_q \subset \mathbb{R}^E$ such that the map $f|_{A_q}^{B_q}$ is a diffeomorphism.
    We note that this now implies that the set $\mathcal{S}$ is a discrete set, since $r \notin A_q$ for any two distinct points $q,r \in \mathcal{S}$ (a consequence of each $f|_{A_q}^{B_q}$ being a diffeomorphism). Because of this, we may suppose we chose our neighbourhoods $A_q$ so that $A_q \cap A_r =\emptyset$ for all distinct points $q,r \in \mathcal{S}$.
    Now let $(p_n)_{n \in \mathbb{N}}$ be a sequence of generic realisations of $G$ that converge to $p$.
    Note that for each $q \in \mathcal{S}$,
    there exists $N_q \in \mathbb{N}$ such that $f(p_n) \in B_q$ for each $n \geq N_q$.
    Furthermore, for each $n \geq N_q$,
    there exists exactly one point $q_n \in A_q$ with $f(q_n) = f(q)$.
    For each $n \in \mathbb{N}$, we now fix $Q_n$ to be the set of such realisations $q_n$.
    By \Cref{prop:genericrealisationnumber},
    there exists $C_G = c_G$ such that $|Q_n| \leq C_G$ for all $n \in \mathbb{N}$.
    As each $q_n$ lies in a unique open neighbourhood $A_q$ for a unique point $q \in \mathcal{S}$,
    it follows that $|\mathcal{S}| \leq C_G$.

    Now assume that $(G,p)$ is not minimally infinitesimally $g$-rigid.
    Let $\mathcal{G}$ to the set of all minimally $g$-rigid spanning subgraphs of $G$,
    and for each $H \in G$,
    fix $\mathcal{S}_H$ to be the set given in \cref{eq: InfRigidModEuclideanMotionsg} with $G$ replaced by $H$.
    Then $\mathcal{S}$ is contained in $\bigcup_{H \in \mathcal{G}} \mathcal{S}_H$.
    The result now follows by fixing $C_G = \sum_{H \in \mathcal{G}} C_H$.
\end{proof}

\section{Graphical Erd\H{o}s problem for typical normed spaces}\label{sec:norm}

In this section we prove the following result.

\mainNormedResult*

We actually prove a stronger `pinned' version of the above result. This means we can guarantee the above number of realisations with a vertex pinned to a single point in $P$ i.e. $p(v_0)$ is fixed.

The proof of \Cref{mainNormedresult} relies on the following result proved later in the section.

\begin{theorem}\label{pinnedmostnorms}
 The following is true for most $d$-norms $\|\cdot\|$. For any finite point set $P \subseteq \mathbb R^d$, there exists a point $x\in P$, which determines $\Omega\left(\frac{|P|}{(\log |P|)^2}\right)$ distances to the other points of $P$ with respect to $\|\cdot\|$,
 where the explicit constant depends only on the norm $\|\cdot\|$.
\end{theorem}

We also require the following lemma.
The proof follows a two step method:
(i) we first find a pin with many distances; (ii) we then remove that pin and reapply our result to the remaining points. 
The details of this are given in \cite[Lemma 4.1]{ip18} in the case where $\|\cdot\|$ is the Euclidean norm, and the proof for any $d$-norm is identical.

\begin{lemma}\label{lem:ManyRichPins}
    Fix a $d$-norm $\|\cdot\|$ which has the property that in any finite set $P$ of size $n$ we can find a point $x$ in $P$ so that the set
    \begin{equation*}
        \Delta_{x}(P) := \left\{ \|x-y\| : y \in P  \right\}
    \end{equation*}
    contains at least $H(n)$ distinct distances,
    where $H$ is an increasing function in $n$. Then we can find a set $P' \subseteq P$ so that $|P'| \geq |P|/2$, and for all points $x$ in $P'$ we have that $|\Delta_{x}(P)| \geq H(n/2)$.
\end{lemma}

\begin{proof}[Proof of \Cref{mainNormedresult}]
    We note that, since $G$ is connected, it suffices to suppose that $G$ is a tree.
    Fix a $d$-norm $\|\cdot\|$ where \Cref{pinnedmostnorms} holds, and let
    \[H(n) = C\left(\frac{n}{\log^2(n)}\right),\]
    where $C$ is the explicit constant in the statement of \Cref{mainNormedresult}.
    Pick a node $v_0$ in $G$ and declare this node the root of the tree $G$.
    We now say that a vertex $v$ has depth $i$ if the distance between $v$ and $v_0$ is $i$ with respect to the shortest-path metric for $G$,
    and we set $V_i \subset V$ to be the set of depth $i$ vertices.
    Given $t$ is the maximum depth of any of the vertices of $G$,
    we have that the sets $\{v_0\} = V_0,\ldots,V_t$ partition $V$.

    The idea of our argument now is to build the tree out of rich pins provided by \Cref{lem:ManyRichPins}. 
    We have to be a little careful in how we do this to ensure that we have pins at every depth. Thus, we will create the sets $P_{t}, P_{t-1}, \ldots, P_0$ iteratively so that we choose the vertices at depth $i$ from $P_i$.
    First, we set $P_{t} = P$.
    Next,
    we iteratively define the set $P_i$ from the set $P_{i+1}$ as follows:
    using \Cref{lem:ManyRichPins},
    we fix $P_i$ to be the set of at least $|P_{i+1}|/2$ points $x \in P_{i+1}$ where $|\Delta_x(P_{i+1})| \geq H(|P_{i+1}|/2)$.
    If $|P|$ is sufficiently large,
    each of the sets $P_0 \subseteq \cdots \subseteq P_t$ are non-empty.

    For each $i \in \{0,\ldots,t-1\}$,
    choose a point $x_i \in P_i$ which minimises  $|\Delta_{x_i}(P_{i+1})|$.
    We now restrict our realisations in $P$ to the set of realisations
    \begin{equation*}
        X := \left\{ p \in P^V : p(v_0) = x_0, ~ p(v) \in P_i \text{ if $v$ has depth $i$} \right\}.
    \end{equation*}
    Using the logic that any edge from a vertex of depth $i$ to a vertex of depth $i+1$ has at least $H(|P|/2^{t-i})$ possible lengths it can take, we have
    \begin{equation}\label{eq:mainNormedresult}
        \left| f_{\|\cdot\|,G}(X) \right| \geq \prod_{i=1}^{t-1} \prod_{v \in V_i} |\Delta_{x_i}(P_{i+1})| \geq \prod_{i=1}^{t-1}H(|P|/2^{t-i})^{|V_i|} \geq (H(|P|/2^t))^{|V|-1}.
    \end{equation}
    Unpacking the function $H(n) = C\left(\frac{n}{\log^2(n)}\right)$, we have that
    \begin{equation*}
        \left|f_{\|\cdot\|, G}\left(P^V\right)\right| 
        \geq \left|f_{\|\cdot\|, G}\left(X\right)\right| 
        \geq \left(\frac{C|P|}{2^t(\log(|P|/2^t))^2}\right)^{|V|-1}
        \gg \left(\frac{|P|}{(\log |P|)^2}\right)^{|V|-1},
    \end{equation*}
    as claimed.
\end{proof}

\subsection{A graph colouring lemma}

Our method for proving \Cref{pinnedmostnorms} is to convert the problem into an edge colouring problem for complete graphs.
With this in mind, we now prove the edge colouring results needed for the proof. We recall that a function $f : U \subset \mathbb{R} \rightarrow \mathbb{R}$ is \emph{subadditive} if $f(x+y) \leq f(x) + f(y)$ for all $x,y \in U$.

\begin{lemma}\label{lem:generalcolour}
    For fixed $\alpha \in [0,1]$, let $f : (\alpha,\infty) \rightarrow (0,\infty)$ be a strictly increasing invertible function, and let $g: [1,\infty) \rightarrow (0,\infty)$ be a subadditive function such that $f^{-1}(t)\geq g(t)$ for all $t \geq 1$.
    Suppose that the edges of the complete graph $K_n$ are coloured such that any monochromatic subgraph on $m \geq 1$ vertices contains at most $f(m)$ edges. Then there exists a vertex of $K_n$ touching at least $\frac{1}{n}g(\binom{n}{2})$
    different colour classes.
\end{lemma}

\begin{proof}
Fix $C$ to be the set of edge colours for $K_n$,
and for each colour $c \in C$ we denote by $e_c$ the number of edges with colour $c$ and by $n_c$ the number of vertices contained in an edge of colour $c$.
We calculate the expectation of the number of colours touching a vertex. We have
\begin{align*}
&\mathbb{E}_{v \in V}[\#\text{colour classes touching }v] \\ 
& = \frac{1}{n}\sum_{v \in V}\sum_{c \in C}\mathbf{1} (v\text{ is contained in an edge of colour }c) 
 \\
&= \frac{1}{n} \sum_{c \in C} n_c \geq \frac{1}{n} \sum_{c \in C} f^{-1}(e_c) \geq \frac{1}{n} \sum_{c \in C} g(e_c)\geq  \frac{1}{n}g\left (\binom{n}{2}\right ),
\end{align*}
where the first inequality follows from the assumption that $f$ is strictly monotone increasing and $e_c \leq f(n_c)$, the second follows from the assumption on $f$ and $g$, and the last follows from the subadditivity of $g$.
\end{proof}

We are now ready for our key proposition for the section.

\begin{prop}\label{prop:pinnednormcolour}
    Let $d \geq 2$ be an integer and suppose that the edges of $K_n$ are coloured such that any monochromatic subgraph on $m$ vertices contains at most $\frac{d}{2} m \log(m)$ edges. Then there exists a vertex that is incident with edges from $\Omega(\frac{n}{\log n})$ colour classes.
\end{prop}

\begin{proof}
    Let 
    \begin{equation*}
        f : (1,\infty) \rightarrow (0,\infty), ~ t \mapsto \frac{d}{2}t \log(t), \qquad g : [1,\infty) \rightarrow (0,\infty), ~ t \mapsto \frac{2t}{d\log(dt +1)}.
    \end{equation*}

    \begin{claim}\label{claim:inverse} For $m> 1$ we have $f^{-1}(m)> g(m)$.
    \end{claim}
    \begin{proof} For $y> 1$ let $x$ be such that $y = \frac{d}{2} x \log(x)$. Note that we have $x >1$.
    Then
    \begin{equation*}
        g(y) = \frac{2y}{d\log(dy +1)} = \frac{2\frac{d}{2} x \log(x)}{d\log(d\frac{d}{2} x \log(x) +1)} < \frac{x \log(x)}{\log(d\frac{d}{2} x \log(x))} = \frac{x \log(x)}{\log(x) + \log(\frac{d^2}{2} \log(x))} < x,
    \end{equation*}
    and so $f^{-1}(y) = x > g(y)$.
    \end{proof}

    \begin{claim}\label{claim:subadd} 
        $g(m)$ is subadditive on $[1,\infty)$.
    \end{claim}

    \begin{proof}
        Choose any $a,b \in [1,\infty)$.
        Then, as the log function is strictly increasing, we have that
        \begin{align*}
            &\frac{d}{2}(g(a) + g(b) - g(a+b)) = \frac{a}{\log(da+1)}+\frac{b}{\log(db+1)} - \frac{a+b}{\log(d(a+b)+1)} \\
             &= \frac{a\log(db+1)\big(\log(d(a+b)+1) - \log(da+1)\big) + b\log(da+1)\big(\log(d(a+b)+1) - \log(db+1)\big)}{\log(da+1)\log(db+1)\log(d(a+b)+1)}\\
             &> 0.
        \end{align*}
        Thus, $g$ is subadditive.
    \end{proof}

     By \Cref{claim:inverse} and \Cref{claim:subadd} we can use \Cref{lem:generalcolour} with our chosen functions $f,g$. 
     We obtain that there exists a vertex $v$ of $K_n$ touching 
     \begin{equation*}
    \frac{1}{n}g\left (\binom{n}{2}\right )=\frac{1}{n}\frac{\binom{n}{2}}{d \log\left (d\binom{n}{2} +1\right )}=\Omega \left (\frac{n}{\log(n)}\right)
    \end{equation*}
    different colour classes. 
\end{proof}

\subsection{Proof of \texorpdfstring{\Cref{pinnedmostnorms}}{key proposition}}

In order to prove \Cref{pinnedmostnorms}, we recall the following result of Alon, Buci\'{c}, and Sauermann originally stated in \Cref{subsec:intro:norm}.

\ABS*

We now combine \Cref{ABS} with \Cref{prop:pinnednormcolour} to prove \Cref{pinnedmostnorms}.

\begin{proof}[Proof of \Cref{pinnedmostnorms}]
    We reinterpret this geometric distance problem as a problem about coloured graphs; 
    the $n=|P|$ points of $P$ will correspond to the vertex set of $K_n$, and the colour of an edge of $K_n$ is simply the distance between the two endpoints in $P$ (and hence a colour class will correspond to all edges of the same length). 
    The number of occurrences of any colour is at most $\frac{d}{2}n \log n$ by \Cref{ABS}. 
    The result follows from \Cref{prop:pinnednormcolour}.
\end{proof}

\subsection{The weak pinned distance conjecture}

Our colouring argument has an additional use for distance problems with the Euclidean distance metric.
We begin by stating the following conjecture which first appeared in \cite{erdHos1990some}.

\begin{conjecture}[Weak pinned distance conjecture]
    For any $\varepsilon > 0$,
    there exists a constant $C_\varepsilon>0$ such that the following holds.
    For any finite point set $P \subseteq \mathbb R^2$, there exists a point $q\in P$, which determines at least $C_{\varepsilon} |P|^{1-\varepsilon}$ distances to the other points of $P$ with respect to the Euclidean distance.
\end{conjecture}

A more famous conjecture is the weak unit distance conjecture, first stated by Erd\H{o}s \cite{erdos1946sets} in 1946. The best known bound is due to Spencer, Szemer\'edi and Trotter \cite{SpencerUnit} in 1984.

\begin{conjecture}[Weak unit distance conjecture]\label{conj:udc}
    For any $\varepsilon > 0$,
    there exists a constant $C_\varepsilon>0$ such that the following holds.
    For any finite point set $P \subseteq \mathbb R^2$, there exists at most $C_{\varepsilon} |P|^{1+\varepsilon}$ distances of length 1 with respect to the Euclidean distance.
\end{conjecture}

We now prove that the latter implies the former.
To show this, we need the following variant of \Cref{prop:pinnednormcolour}.

\begin{prop}\label{prop:pinnedeuclideancolour}
    Let $d \geq 2$ be an integer and $C>0$ a fixed scalar, and suppose that the edges of $K_n$ are coloured such that any monochromatic subgraph on $m$ vertices contains at most $C m^{1+\delta}$ edges. Then there exists a vertex $v$ such that $v$ has edges from $\Omega(n^{1-\varepsilon})$ colour classes,
    where $\varepsilon = \frac{2\delta}{1+\delta}$.
\end{prop}

\begin{proof} 
    Let $f(m)= Cm^{1+\delta}$ and $g(m)=f^{-1}(m)=C^{-1/(1+\delta)} m^{1/(1+\delta)}$. Since $g(m)$ is subadditive, we can apply \Cref{lem:generalcolour}, and obtain that there exists a vertex of $G$ touching
    \begin{equation*}
       \frac{1}{n}g\left (\binom{n}{2}\right )=\Omega(n^{1- 2\delta/(1+\delta)})=\Omega (n^{1-\varepsilon})
    \end{equation*}
    different colour classes. 
\end{proof}

\begin{theorem}\label{euclideanpinned}
    The weak unit distance conjecture implies the weak pinned distance conjecture.
\end{theorem}

\begin{proof}
    We first fix $\delta > 0$ such that $\varepsilon = \frac{\delta (2 + \delta)}{1+\delta}$.
    We begin by reinterpreting this geometric distance problem as a problem about coloured graphs; 
    the $n=|P|$ points of $P$ will correspond to the vertex set of the complete graph $K_n$, and the colour of an edge of $K_n$ is simply the distance between the two endpoints in $P$ (and hence a colour class will correspond to all edges of the same length). By the weak unit distance conjecture the number of occurrences of any colour in a subgraph on $m$ vertices is at most $C_{\varepsilon}m^{1+\delta}$.
    Thus, the result now holds by \Cref{prop:pinnedeuclideancolour}.
\end{proof}

Interestingly, this all ties back to the graphical version of the Erd\H{o}s distinct distance problem described in \Cref{prob:maineuclidproblem}.
Using a similar proof technique as was described in \Cref{mainNormedresult},
Iosevich and Passant \cite{ip18} proved that the weak pinned distance conjecture implies the following conjecture. 

\begin{conjecture}\label{Conj: GraphDistanceSharp}
    Let $P$ be a finite set in $\mathbb{R}^2$ and suppose that $G=(V,E)$ is a finite, simple, connected graph. 
    Then, for all $\varepsilon>0$, there exists some positive constant $C_\varepsilon$ independent of the choice of $P$ so that
    \[\Big|f_G\left(P^V\right)\Big| \geq C_\varepsilon|P|^{|V|-1-\varepsilon}.\]
\end{conjecture}

It now follows from \Cref{euclideanpinned} that the weak unit distance conjecture also implies \Cref{Conj: GraphDistanceSharp}.

\section{What happens if the rigidity condition is dropped?}\label{sec:flexible}

In this section we consider relaxing the rigidity assumption. We will focus only on the setting of distance constraints in the Euclidean plane for $g$-rigidity in this section. As a result, for simplicity, we will refer to $f_{g,G}$ by $f_G$ in this section.
It would be interesting to extend the results below to the more general context, but that is left open for the interested reader.

\subsection{Lower bounds for flexible graphs}

For special families of flexible graphs, similar results to \Cref{thm:Hypergraph2D} can be obtained with a slightly weaker lower bound. 
In particular, Rudnev \cite{RudnevHinge} proved that if $G$ is a path with three vertices,
then for any finite set $P \subset \mathbb{R}^2$,
we have
\begin{equation}\label{eq:rudnevhinge}
    \Big| f_{G}(P^3)\Big| \gg \frac{|P|^2}{(\log |P|)^3}.
\end{equation}
Rudnev's bound is almost tight, as if $P$ is a square grid then 
\begin{equation*}
    \Big| f_{G}(P^3)\Big| \ll \frac{|P|^2}{\log |P|}.
\end{equation*}
Passant \cite{passantChains} later extended Rudnev's result to paths of arbitrary length.

\begin{theorem}[Passant \cite{passantChains}]\label{thm:passantChains}
    Let $G$ be a path with $n$ vertices.
    Then for any finite set $P \subset \mathbb{R}^2$,
    \begin{equation}\label{eq:passantchain}
        \Big| f_{G}(P^n)\Big| \gg \frac{|P|^{n-1}}{(\log |P|)^{\frac{13}{2}(n-2)}}.
    \end{equation}
\end{theorem}

We now extend this using the following concept:
a \emph{0-extension} of a graph $G$ is an operation which creates a new graph $G'$ from $G$ by adding one new vertex, and connecting the new vertex to two vertices of $G$. 
Doing so allows us to prove the following two results.

\begin{theorem}\label{0extthm}
    Let $G = (V,E)$ be a graph that is obtained from a path with $k$ vertices, using a sequence of $t$ 0-extensions. 
    Then for any finite set $P \subset \mathbb{R}^2$,
    \[ \Big| f_{G}(P^V) \Big| \gg \frac{|P|^{k-1 + t}}{(\log|P|)^{\frac{13}{2}(k-2)}}.\]
\end{theorem}

\begin{theorem}\label{thm:bipartite}
    Let $K_{2,n}$ be the complete bipartite graph with parts of size 2 and $n$. 
    Then for any finite set $P \subset \mathbb{R}^2$,
    \[ \Big| f_{K_{2,n}}(P^{n+2}) \Big| \gg \frac{|P|^{n+1}}{(\log|P|)^3}.\]
\end{theorem}

It will be useful for us to work with injective realisations of $G$.
\begin{definition}
    Let $G=(V,E)$ be any graph and $P$ a finite subset of $\mathbb{R}^2$. We denote by $P^V_{\text{inj}}$ the injective realisations of $G$ on $P$ i.e.
    \[ P^V_{\text{inj}} = \left\{ p \in P^V: p \text{ injective} \right\}.\]
\end{definition}
The following lemma implies that if $G$ has many inequivalent injective realisations within a point set $P$, then after performing a 0-extension there are still many inequivalent injective realisations.
\begin{lemma}\label{lem:injectiveRealisations}
    Suppose that $G=(V,E)$ is any finite graph. Suppose that $G'=(V',E')$ is a graph that is formed from $G$ by a 0-extension. Then,
    \[ \Big|f_{G'}\left(P^{V'}_{\text{inj}}\right)\Big| \gg  \Big|f_{G}\left(P^V_{\text{inj}}\right)\Big|\big|P\big|.\]
\end{lemma}

\begin{proof}
    Each realisation $p$ in $P^V_{\text{inj}}$ can be extended to a realisation of $G'$, by choosing any point in $P \setminus p(V)$ as the new vertex $v$, which is connected to $u$ and $w$. The key observation is that for each such newly created injective realisation, there is only at most one other new realisation equivalent to it, namely the one given by reflecting the position of $v$ in the line connecting $p(u)$ and $p(w)$. This is shown in \Cref{fig:audie}; the blue injective realisation has been extended to a larger realisation, with the red and orange edges showing two possible equivalent extensions. Performing this for each injective realisation gives
    \[ \Big|f_{G'}\left(P^{V'}_{\text{inj}}\right)\Big|\geq \Big|f_{G}\left(P^V_{\text{inj}}\right)\Big|\left( \frac{|P|-|V|}{2}\right) \gg  \Big|f_{G}\left(P^V_{\text{inj}}\right)\Big|\big|P\big|\]
    as needed.
\end{proof}    

\begin{figure}[ht]  
\begin{minipage}{0.45 \textwidth}
\centering
    \begin{tikzpicture}
        \draw[blue, ultra thick] (0.5,0) -- (1,0);
    \draw[blue, ultra thick] (0.5,0) -- (0,0.75);
    \draw[blue, ultra thick] (0,0.75) -- (1,0);
    \draw[blue, ultra thick] (0.5,0) -- (1,0);
    \draw[blue, ultra thick] (1.5,1) -- (1,0);
    \draw[blue, ultra thick] (1.5,1) -- (2,-1);
    \draw[blue, ultra thick] (0.5,0) -- (2,-1);
    \draw[blue, ultra thick] (0.5,0) -- (-1,0);
    \draw[blue, ultra thick] (0,0.75) -- (-1,0);
        \filldraw[black] (-1,0) circle (2pt);
        \filldraw[black] (1,0) circle (2pt);
        \filldraw[black] (-2,-1) circle (2pt);
        \filldraw[black] (-2,1) circle (2pt);
        \filldraw[black] (2,-1) circle (2pt);
        \filldraw[black] (1.5,1) circle (2pt);
        \filldraw[black] (0.5,0) circle (2pt);
        \filldraw[black] (0,0.75) circle (2pt);
        \filldraw[black] (0,-1) circle (2pt);
    \end{tikzpicture}
\end{minipage}
\begin{minipage}{0.45 \textwidth}
\centering
    \begin{tikzpicture}
    \draw[blue, ultra thick] (0.5,0) -- (1,0);
    \draw[blue, ultra thick] (0.5,0) -- (0,0.75);
    \draw[blue, ultra thick] (0,0.75) -- (1,0);
    \draw[blue, ultra thick] (0.5,0) -- (1,0);
    \draw[blue, ultra thick] (1.5,1) -- (1,0);
    \draw[blue, ultra thick] (1.5,1) -- (2,-1);
    \draw[blue, ultra thick] (0.5,0) -- (2,-1);
    \draw[blue, ultra thick] (0.5,0) -- (-1,0);
    \draw[blue, ultra thick] (0,0.75) -- (-1,0);
    \draw[orange, ultra thick] (0,0.75) -- (-2,1);
    \draw[orange, ultra thick] (-2,1) -- (-1,0);
    \draw[red, ultra thick] (0,0.75) -- (0,-1);
    \draw[red, ultra thick] (0,-1) -- (-1,0);
         \filldraw[black] (-1,0) circle (2pt);
        \filldraw[black] (1,0) circle (2pt);
        \filldraw[black] (-2,-1) circle (2pt);
        \filldraw[black] (-2,1) circle (2pt);
        \filldraw[black] (2,-1) circle (2pt);
        \filldraw[black] (1.5,1) circle (2pt);
        \filldraw[black] (0.5,0) circle (2pt);
        \filldraw[black] (0,0.75) circle (2pt);
        \filldraw[black] (0,-1) circle (2pt);
    \end{tikzpicture}
\end{minipage}
\caption{Left: A possible realisation $p \in P^V$.
Right: Two equivalent realisations (with edge colours red-blue and orange-blue respectively) that can be obtained via 0-extension.}
\label{fig:audie}
\end{figure}
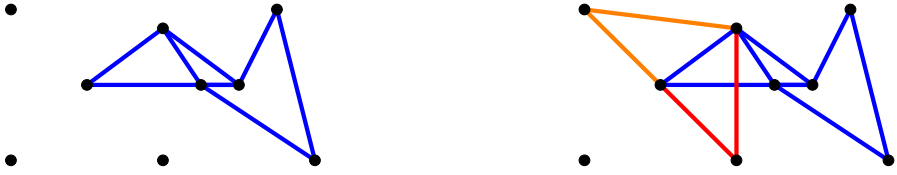 

We can now prove both \Cref{0extthm} and \Cref{thm:bipartite}.

\begin{proof}[Proof of \Cref{0extthm}]
    Using \Cref{lem:injectiveRealisations}, it suffices to show that when $G$ is a path of length $k$, we have $\big|f_G(P^V_{\text{inj}})\big|\gg\frac{|P|^{k-1}}{\log^{\frac{13}{2}(k-2)}|P|}$. Fortunately, we can show this using ideas from \cite{passantChains, RudnevHinge}. Specifically, for a path of length $k$ we have that
    \begin{align*}
        \binom{|P|}{k} &= \sum_{\lambda \in f_G(P^V_{\text{inj}})} \left|\left\{p \in P^V_{\text{inj}} : f_G(p) = \lambda \right\}\right|
        \\ &\leq \Big|f_G(P^V_{\text{inj}})\Big|^{1/2} \left( \sum_{\lambda \in f_G(P^V_{\text{inj}})} \left|\left\{p \in P^V_{\text{inj}} : f_G(p) = \lambda \right\}\right|^2 \right)^{1/2} \quad \text{(by the Cauchy-Schwarz inequality)} \\
        & \leq  \Big|f_G(P^V_{\text{inj}})\Big|^{1/2} \left|\left\{(p,p') \in P^V : f_G(p) = f_G(p') \right\} \right|^{1/2},
    \end{align*}
    and so the result follows from the energy bound for paths used to prove \cite[Theorem 2.1]{passantChains}\footnote{The energy upper bound stated in \cite{passantChains} is actually $|L|^{(k+1)/2}(\log|L|)^{\frac{13}{2}(k-2)}$, where $L$ is the set of lines in $\mathbb{R}^3$ formed from pairs of distinct points in $P$ using the Guth and Katz's line construction \cite[Eq.~(2.3)]{guth2015erdHos}. Our energy upper bound then follows from the simple observation that $|L| = \binom{|P|}{2} \ll |P|^2$.}; namely, that
    \begin{equation*}
        |\{(p,p') \in P^V : f_G(p) = f_G(p') \} | \ll |P|^{k+1}(\log|P|)^{\frac{13}{2}(k-2)}.\qedhere
    \end{equation*}
\end{proof}

\begin{proof}[Proof of \Cref{thm:bipartite}]
    This proof follows in the same way as the above, but instead making use of the energy result for hinges (paths with three vertices) from \cite[Theorem 1]{RudnevHinge}.
\end{proof}

It now follows from \Cref{cor:euclidean} ($n=3$), \Cref{0extthm} and \Cref{thm:bipartite} that if $C_n$ is the cycle with $n \geq 3$ vertices and $P\subset \mathbb{R}^2$,
we now have that
\begin{equation*}
    \Big| f_{C_n} (P^n ) \Big| \gg
    \begin{cases}
        |P|^{2} &\text{if } n =3,\\
        \frac{|P|^{3}}{(\log|P|)^{3}} &\text{if } n =4,\\
        \frac{|P|^{n-1}}{(\log|P|)^{\frac{13}{2}(n-3)}} &\text{if } n \geq 5.
    \end{cases}
\end{equation*}
When $P$ is contained in the integer line, we see that $\big| f_{C_n} (P^n ) \big| \ll |P|^{n-1}$. It is unclear where the lower bound should fit in this interval.

\subsection{Exotic behaviour of flexible graphs}\label{subsec:genus}

In this section we illustrate the importance of the hypothesis that the graph is rigid with 3 or more vertices. Extending our techniques to graphs which are not rigid seems 
to require some serious algebraic number theory.
This is hinted at with the behaviour of the complete graph $K_2$.
Unlike with all other rigid graphs,
the graph $K_2$ gives the bound $\big|f_{K_2}(P^2)\big| \gg (|P|/\log |P|)$,
which is (almost) witnessed on the grid by a lower bound of $\Omega(|P|/\log \sqrt{|P|})$.
This logarithmic loss for the grid stems from an important property of $K_2$:
if we fix one vertex to a point in the grid, then we can often rotate the other vertex in a circle to find other realisations of $K_2$ in the grid with the same edge length.
This property is not true for larger rigid graphs,  however it is known to hold for the hinge, which is a flexible graph with 1 degree of freedom when considered as a linkage.

We now wish to explore the more general family of such flexible graphs.
To do so, we first need to complexify our framework model.
The complexification of the measurement map $f_G$ given in \cref{eq:measurementmap} is the polynomial map
\begin{equation*}
    h_G : (\mathbb{C}^2)^V \rightarrow \mathbb{C}^E, ~ p \mapsto \left( \big( p_1(v) - p_1(w)  \big)^2 + \big( p_2(v) - p_2(w)  \big)^2 \right)_{vw \in E}.
\end{equation*}
We similarly complexify $\SE$ to the connected algebraic group $\SE_\mathbb{C}$ by allowing any complex map $x \mapsto A x +b$ where $A$ is a complex $2 \times 2$ matrix satisfying $A^T A = I_2$ and $\det A=1$, and $b \in \mathbb{C}^2$.

Whenever $(G,p)$ is a generic framework, our fibre $h_G^{-1}(h_G(p))$ is a smooth algebraic set.
Because of this, we can begin to look at various invariant properties of the space $h_G^{-1}(h_G(p))$, particularly those that are invariant to our choice of generic $p$.
The invariant that we now wish to focus on is the genus of $h_G^{-1}(h_G(p))$, or more accurately, the genus of the quotient space $h_G^{-1}(h_G(p))/\SE_\mathbb{C}$.

To describe this concept in more detail, we require the following technical lemma.
Here we recall that the \emph{genus} of a smooth algebraic curve is equal to the geometric genus of any smooth projective closure of the curve when considered as a Riemann surface\footnote{That this definition is well-defined follows from three classical facts from algebraic geometry: (i) every smooth algebraic set has a smooth projective closure; (ii) any affine variety is birationally equivalent to its projective closure; (iii) any two birationally equivalent smooth projective algebraic curves are isomorphic (and hence are homeomorphic as surfaces)}.
We also recall that a polynomial map $f : \mathbb{C}^m \rightarrow \mathbb{C}^n$ is \emph{dominant} if $f(\mathbb{C}^m)$ is dense in $\mathbb{C}^n$.
The proof of the following lemma can be found in \Cref{appendix}.

\begin{restatable}{lemma}{Genus}\label{lem:generalfibre}
    Let $f : \mathbb{C}^{n+1} \rightarrow \mathbb{C}^n$ be a dominant polynomial map.
    Then there exists a non-empty Zariski open set $U \subset \mathbb{C}^{n+1}$ and constants $k,g$ such that for each $p \in U$ the Jacobian ${\rm J} f(p)$ has rank $n$ and the set $f^{-1}(f(p))$ consists of $k$ disjoint smooth algebraic curves,
    each having genus $g$.
\end{restatable}

The following definition can now be seen to be well-defined as a consequence of \Cref{lem:generalfibre}.

\begin{definition}\label{def:flexgenus}
    We say a graph $G$ with 3 or more vertices has \emph{1 degree of freedom} (in $\mathbb R^2$) if $|E|= 2|V|-4$ and for any generic realisation $p$ in $\mathbb{R}^2$,
    we have $\rank \mathrm{J} f_G (p) = 2|V| - 4$.
    The \emph{flex-genus} of a graph with 1 d.o.f.~is the topological genus of any irreducible component of the smooth algebraic curve $C_{G,vw}(p)$,
    defined to be the algebraic set
    \begin{equation*}
        C_{G,vw}(p) := h_G^{-1}(h_G(p)) \cap \{ q \in (\mathbb{C}^2)^V : q(v) = p(v), ~ q(w) = p(w) \}
    \end{equation*}
    where $vw$ is an arbitrary edge of $G$ and $p$ is any realisation of $G$ contained in the Zariski open set $U \subset (\mathbb{C}^2)^V$ guaranteed by \Cref{lem:generalfibre}.
\end{definition}

Flex-genus was recently investigated by Schicho, Tewari and Warren in \cite{schicho2024genusdegreefreedomplanar}.
Some examples of 1 d.o.f.~graphs and their corresponding flex-genus can be seen in \Cref{fig:flexgenusexamples}.

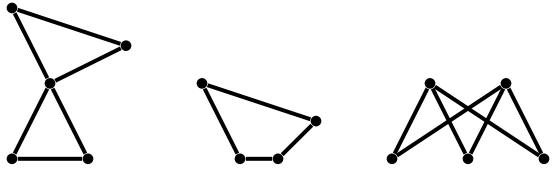
\begin{figure}[ht]
    \centering
    \begin{tikzpicture}[scale=0.5]
        \begin{scope}
            \node[vertex] (1) at (-1,0) {};
            \node[vertex] (2) at (1,0) {};
            \node[vertex] (3) at (0,2) {};
            \node[vertex] (4) at (2,3) {};
            \node[vertex] (5) at (-1,4) {};
            
            \draw[edge] (1) -- (2);
            \draw[edge] (2) -- (3);
            \draw[edge] (1) -- (3);
            \draw[edge] (4) -- (5);
            \draw[edge] (4) -- (3);
            \draw[edge] (5) -- (3); 
        \end{scope}
        \begin{scope}[xshift=5cm]
            \node[vertex] (1) at (2,1) {};
            \node[vertex] (2) at (1,0) {};
            \node[vertex] (3) at (-1,2) {};
            \node[vertex] (4) at (0,0) {};
            
            \draw[edge] (4) -- (2);
            \draw[edge] (3) -- (1);
            \draw[edge] (2) -- (1);
            \draw[edge] (4) -- (3);
        \end{scope}
        \begin{scope}[xshift=9cm]
            \node[vertex] (1) at (0,0) {};
            \node[vertex] (2) at (2,0) {};
            \node[vertex] (3) at (4,0) {};
            \node[vertex] (4) at (1,2) {};
            \node[vertex] (5) at (3,2) {};
            
            \draw[edge] (1) -- (4);
            \draw[edge] (2) -- (4);
            \draw[edge] (3) -- (4);
            \draw[edge] (1) -- (5);
            \draw[edge] (2) -- (5);
            \draw[edge] (3) -- (5);
        \end{scope}
    \end{tikzpicture}
    \caption{Three 1 d.o.f.~graphs with (from left to right) flex-genus 0, 1 and 5.}
    \label{fig:flexgenusexamples}
\end{figure}

We now say that any realisation $p$ of $G$ contained in the Zariski open set $U$ mentioned in \Cref{def:flexgenus} is \emph{general}.
The properties of the set $C_{vw}(p)$ when $p$ is a general realisation are summarised in the following result.

\begin{lemma}\label{lem:rationalpoints1dof}
    Let $G$ be a 1 d.o.f.~graph with an edge $vw \in E$ and at least 4 vertices.
    Let $k$ and $g$ be the number of connected components and the genus of $C_{vw}(p)$ for any choice of general $p$.
    Then for each general $p$, the quotient space $h_G^{-1}(h_G(p)) / \SE_\mathbb{C}$ is a 1-dimensional complex manifold consisting of $k$ disjoint connected components of genus $g$.
    Furthermore; given general $p \in (\mathbb{Q}^2)^V$ and the set
    \begin{equation*}
        \mathbb{Q}_G(p) := \Big\{ [q] \in h_G^{-1}(h_G(p)) / \SE_\mathbb{C} : [q] \cap (\mathbb{Q}^2)^V \neq \emptyset  \Big\},
    \end{equation*}
    we have
    \begin{equation*}
        \big|\mathbb{Q}_G(p) \big| = \Big| C_{G,vw}(p) \cap (\mathbb{Q}^2)^V \Big|.
    \end{equation*}
\end{lemma}

\begin{proof}
    As $p$ is general, $\dim \ker {\rm J} h_G (q)=4$ for each $q \in h_G^{-1}(h_G(p))$ and the algebraic set $h_G^{-1}(h_G(p))$ is smooth.
    We note here that the set $h_G^{-1}(h_G(p))$ is $\SE_\mathbb{C}$-invariant, in that the group action of $\SE_\mathbb{C}$ on $h_G^{-1}(h_G(p))$ is performed by morphisms of algebraic sets.
    Since $G$ has at least 4 vertices and is a 1 d.o.f.~graph,
    no realisation $q \in h_G^{-1}(h_G(p))$ can have colinear vertices.
    Indeed, if one such $q$ did have colinear vertices, then $\dim \ker {\rm J} h_G (q) \geq |V| + 1 >4$, contradicting that $\dim \ker {\rm J} h_G (q)=4$ for each $q \in h_G^{-1}(h_G(p))$.
    It hence follows that our group action is also free (in that $\theta p = p$ if and only if $\theta$ is the identity map).
    Since $\SO(2)_\mathbb{C}$ is a reductive linear algebraic group\footnote{A linear algebraic group is reductive if and only if it contains no closed normal subgroup isomorphic to the additive group $\mathbb{C}^n$ for any $n \geq 1$ \cite[Theorem 1.23]{brion}. As the proper closed subgroups of $\SO(2)_\mathbb{C}$ are finite and $\SO(2)_\mathbb{C}$ is not isomorphic to any additive group $\mathbb{C}^n$, it follows that $\SO(2)_\mathbb{C}$ is reductive.} and $\SE_\mathbb{C} = \mathbb{R}^2 \rtimes \SO(2)_\mathbb{C}$, a minor alteration of \cite[Proposition 1.2]{brion} shows that there exists an algebraic curve $C$ and a diffeomorphism $\phi : h_G^{-1}(h_G(p)) / \SE_\mathbb{C} \rightarrow C$ such that, given the quotient map $\pi : h_G^{-1}(h_G(p)) \rightarrow h_G^{-1}(h_G(p)) / \SE_\mathbb{C}$, the composition map $\phi \circ \pi$ is a surjective morphism of algebraic sets.

    Fix $R$ to be the affine reflection that fixes the line through the points $p(v),p(w)$.
    For each $q \in h_G^{-1}(h_G(p))$,
    there exists exactly one other realisation $q' \in h_G^{-1}(h_G(p))$ where $Rq' = q$ (as otherwise $q$ would have colinear vertices).
    As any other non-trivial map contained in $\SE_\mathbb{C}$ does not have both $p(v)$ and $p(w)$ as invariant points,
    the domain-restricted map $\tilde{\pi} : C_{vw}(p) \rightarrow h_G^{-1}(h_G(p)) / \SE_\mathbb{C}$ is a bijection.
    As genus is preserved by rational maps,
    it now follows that $C$ (and hence $h_G^{-1}(h_G(p)) / \SE_\mathbb{C}$) and $C_{G,vw}(p)$ have the same genus.
    For the final claim regarding $\mathbb{Q}_G(p)$,
    we note that if $[q] \in \mathbb{Q}_G(p)$,
    then (by applying rotations and translations to move $v$ to $p(v)$ and $w$ to $p(w)$)
    there exists $q' \in [q] \cap C_{vw}(p)$.
\end{proof}

Our plan is to now combine \Cref{lem:rationalpoints1dof} with the following refinement of Faltings' theorem \cite{Faltings1983,Faltings1984}.

\begin{conjecture}[Uniform boundedness conjecture for rational points]\label{conj:uniformbounded}
    For every genus $g \geq 2$ and every number field $K$, there exists $c>0$ such that any smooth algebraic curve of genus $g$ defined over $K$ has at most $c$ rational points. 
\end{conjecture}

If the uniform boundedness conjecture holds, then for any 1 d.o.f.~graph $G$ with a flex-genus of at least 2, there exists a constant $c>0$ such that the cardinality of the set $\mathbb{Q}_G(p)$ given in \Cref{lem:rationalpoints1dof} is at most $c$ for any choice of general $p$.
We can exploit this fact to show the following result.

\begin{prop}\label{prop:genusgrid}
    Let $G=(V,E)$ be a fixed graph with 1 d.o.f.~and flex-genus at least 2.
    For each square number $n$,
    fix 
    \begin{equation*}
        \mathcal{G}_n := \left\{ \left(x/\sqrt{n},y/\sqrt{n}\right) : x,y \in \left\{0,\ldots,\sqrt{n}\right\} \right\}.
    \end{equation*}
    Suppose that the uniform boundedness conjecture for rational points holds.
    Then
    \begin{equation*}
        \left|f_G\left(\mathcal{G}_n^V\right)\right| \gg n^{|V|-1}.
    \end{equation*} 
\end{prop}

\begin{proof}
    Since $G$ is a 1 d.o.f.~graph with flex-genus at least 2,
    it must have at least 4 vertices;
    indeed, the only 1 d.o.f.~graph with 3 vertices is the hinge which has a flex-genus of 0.
    Let $U \subset (\mathbb{C}^2)^V$ be the Zariski open set of realisations where \Cref{lem:rationalpoints1dof} holds.
    Choose $p \in U$ where (by scaling if necessary) every vertex is contained in the square $[0,1]^2$.
    Then there exists $0 < \varepsilon < 1$ such that,
    if $q \in (\mathbb{R}^2)^V$ satisfies $\|p(v)-q(v)\| < \varepsilon$ for all $v \in V$, then $q \in U$.
    We now fix $\varepsilon$ and define for each $n$ the set
    \begin{equation*}
        S_n := \left\{ q \in (\mathcal{G}_n)^V : \|p(v)-q(v)\| < \varepsilon \text{ for all } v \in V  \right\}.
    \end{equation*}
    By adapting bounds from the Gauss circle problem (for example, \cite{hux}),
    we observe that\footnote{Here we use the convention that $f(n) \sim g(n)$ if and only if $\lim_{n \rightarrow \infty} f(n)/g(n) = 1$.} $|S_n| \sim \pi \varepsilon^2 n^{|V|}$, and so $|S_n| \gg n^{|V|}$.
    For each point $\mathbf{t}$ in $\mathbb{R}^E$, we define its set of representations $\nu(\mathbf{t})$ as
    \[ \nu(\mathbf{t}) = \{ p \in S_n: f_G(p) = \mathbf{t} \}.\]
    Then, we have that
    \[\left(n^{|V|}\right)^2  \ll |S_n|^2 = \left( \sum_{t \in f_G(S_n)}|v(\mathbf{t})|\right)^2.\]
    For each non-empty $\nu(\mathbf{t})$, select one $p \in \nu(\mathbf{t})$ and fix it, calling this fixed element $p_{\mathbf{t}}$.
    As all realisations in $\nu(\mathbf{t})$ have the same edge lengths (by definition), for any realisation $(G, q) \in \nu(\mathbf{t})$ we have that $(G,p_{\mathbf{t}}) \sim (G,q)$. 

    By \Cref{lem:rationalpoints1dof},
    there exists $k,g$ that are independent of $p$ such that each quotient space $h_G^{-1}(\mathbf{t})/ \SE_{\mathbb{C}}$ consists of $k$ smooth irreducible algebraic curves, here labelled $\mathcal{S}_1(\mathbf{t}),\ldots, \mathcal{S}_k(\mathbf{t})$, each with genus $g$.
    Under the assumption of the uniform boundness conjecture in conjunction with \Cref{lem:rationalpoints1dof},
    there exists $c>0$ independent of our choice of $p$ such that each set
    \begin{equation*}
        \mathbb{Q}_i(t) := \left\{ [q] \in \mathcal{S}_i(\mathbf{t}) : [q] \cap (\mathbb{Q}^2)^V \neq \emptyset \right\}
    \end{equation*}
    contains $c_i(\mathbf{t}) \leq c$ points.
    For each $i \in [k]$, define the set
    \begin{equation*}
        \nu_i(\mathbf{t}) := \left\{ q \in \nu(\mathbf{t}) : q \in [q'] \text{ for some } [q'] \in \mathcal{S}_i(\mathbf{t}) \right\}.
    \end{equation*}
    Then for each $i \in [k]$, $\nu_i(\mathbf{t})$ can be partitioned into $c_i(\mathbf{t})$ subsets $\nu_{i,1}(\mathbf{t}), \ldots, \nu_{i,c_i(\mathbf{t})}(\mathbf{t})$ where each $\nu_{i,j}(\mathbf{t})$ contains only congruent realisations.
    Hence,
    \begin{equation*}
        n^{2|V|} \ll \left(\sum_{\mathbf{t}\in f_G(S_n)}|\nu(\mathbf{t})|\right)^2 = \left(\sum_{\mathbf{t}\in f_G(S_n)}\sum_{i=1}^{k} \sum_{j=1}^{c_i(\mathbf{t})}|\nu_{i,j}(\mathbf{t})|\right)^2.
    \end{equation*}
    Upper bounding each $|\nu_{i,j}(\mathbf{t})|$ in this final line with the largest set of representatives, denoted $|\nu_{\max}(\mathbf{t})|$, gives us
        \[ n^{2|V|} \ll \left(\sum_{\mathbf{t}\in f_G(S_n)}  c_i(\mathbf{t}) k |\nu_{\max}(\mathbf{t})|\right)^2 \leq \left(\sum_{\mathbf{t}\in f_G(S_n)}  c k |\nu_{\max}(\mathbf{t})|\right)^2.\]
    By applying the Cauchy-Schwarz inequality we obtain
    \begin{equation*}
            n^{2|V|} \ll c^2k^2 \cdot\Big|f_G(S_n)\Big|  \sum_{\mathbf{t}\in f_G(S_n)}|\nu_{\max}(\mathbf{t})|^2.
    \end{equation*}
    As $S_n \subseteq P^V$, we have that $\big|f_G(S_n)\big| \leq \big|f_G\left(\mathcal{G}_n^V\right)\big|$. As $ck = O_{|P|}(1)$, the constant can be suppressed in the asymptotic notation. Putting this together gives
    \begin{equation}\label{eq:genus1}
        n^{2|V|} \ll \Big|f_G\left(\mathcal{G}_n^V\right) \Big|  \sum_{\mathbf{t}\in f_G(S_n)}|\nu_{\max}(\mathbf{t})|^2.
    \end{equation}
    We observe now that
    \begin{equation}\label{eq:genus2}
        \sum_{\mathbf{t}\in f_G(S_n)}|\nu_{\max}(\mathbf{t})|^2 \leq \big|\left\{(p,q) \in S_n\times S_n: \exists \theta \in \SE, ~ \theta p = q.\right\}\big| \leq |\mathcal{E}_{V,\SE}(\mathcal{G}_n)|.
    \end{equation}
    By combining \eqref{eq:genus1} and \eqref{eq:genus2},
    we have that
        \begin{equation*}
            n^{2|V|} \ll \Big|f_G(\mathcal{G}_n^V) \Big| \big|\mathcal{E}_{V,\SE}(\mathcal{G}_n)\big|.
        \end{equation*}    
    The result now follows from \Cref{lem:rotations}.
\end{proof}

So, why is \Cref{prop:genusgrid} interesting? A common heuristic in the study of Erd\H{o}s distance problems is that the correct lower bound for $\big| f_G(P^V)\big|$ is witnessed when $P$ is a grid.
Indeed, it was shown by Mansfield and Passant \cite{MP2024} that when $G=K_3$, a point set satisfies $\big| f_G(P^V)\big| = O(|P|^2)$ only if a positive proportion of $P$ is contained in either a grid-like structure or a circle.
By \Cref{thm:Hypergraph2D}, the circle will always give a lower bound of $\big| f_G(P^V) \big| \gg |P|^{|V|-1}$ so long as $G$ is connected,
and so point sets contained within circles will only be a witness for the lower bound of $\big| f_G(P^V) \big|$ when $\big| f_G(P^V) \big| \gg |P|^{|V|-1}$ holds for $P \subset \mathbb{R}^2$.
The grid seems like a much more likely witness for the lower bound of $\big| f_G(P^V) \big|$,
as can be seen to (almost) hold when $G=K_2$ or $G$ is a path.
If we assume that grids are the correct witnesses for all graphs (rigid or not), then \Cref{prop:genusgrid} suggests that weaker graph properties than rigidity (in this case, a flex genus of at least 2) can guarantee a lower bound of $\big| f_G(P^V) \big| \gg |P|^{|V|-1}$.

\section{Concluding remarks}\label{sec:conclude}

\subsection{Higher dimensions}

A major road-block to extending \Cref{cor:euclidean}, \Cref{cor:dothigher} and \Cref{cor:skewsymm} to higher dimensions is the absence of an analogous result to \Cref{thm: GuthKatz} regarding energy.
Without this, our results in these cases is rather limited.

Two cases where this is not an issue are $\ell_p$ distances ($p \geq 4$) and the `symmetric tensor completion' metric given by \cref{eq:sym_tensor_g}.
For both of these, the only limiting factor is understanding the behaviour of point sets contained in low degree curves.
While the latter case seems likely to be difficult,
the former seems like it could be approachable if we are able to prove three things:
\begin{itemize}
    \item If a graph is rigid in $\ell_p^d$, then it is rigid with respect to the $\ell_p$ metric when restricted to any irreducible algebraic variety contained in $\mathbb{R}^d$. Surface rigidity has been previously investigated \cite{nixowenpower}, however such a result is not even known for the case where $p=2$ (i.e., the Euclidean case).
    \item An analogue of \Cref{thm:energyg} for point sets contained within algebraic sets,
    assuming the graph is rigid with respect to being constrained to the hypersurface.
    \item Extending \Cref{thm:Hypergraph2D} to allow for the curve $C$ to sit in a higher dimensional space.
\end{itemize}
If all of the above are proven, it seems likely that some type of inductive argument could be implemented to prove the following conjecture:

\begin{conjecture}\label{con:lphigher}
    Let $G=(V,E)$ be a graph that is rigid in $\ell_p^d$ for even $p \geq 4$ and $d\geq 1$,
    and let $g$ be the $\ell_p$ metric given in \cref{eq:lp} for the chosen $p,d$.
    Then for any finite set $P \subset \mathbb{R}^d$,
    \begin{equation*}
       \Big| f_{g,G}\left(P^V\right) \Big| \gg |P|^{|V|-1},
    \end{equation*}
    and this bound is tight.
\end{conjecture}

It is possible that using a similar apporach as given above for the dot product,
we could provide a full solution to \Cref{cor:dothigher}.

\begin{conjecture}\label{con:dot}
    Let $G=(V,E)$ be a semisimple graph with $|V| \geq 4$ that is locally completable in $\mathbb{R}^3$,
    and let $g$ be the dot product given in \cref{eq:dot} with $d=3$.
    Then for any finite set $P \subset \mathbb{R}^3$,
    we have
    \begin{equation*}
       \Big| f_{g,G}\left(P^V\right) \Big|  \gg |P|^{|V|-1},
    \end{equation*}  
    and this bound is tight.
\end{conjecture}

\subsection{Rational measurement maps}

Many of the ideas used throughout \Cref{sec:graphicalgrigid} seem to allow for the polynomial $g :(\mathbb{R}^d)^k \rightarrow \mathbb{R}$ to be replaced with a rational function $g = g_1/g_2$ for some polynomials $g_1,g_2 :(\mathbb{R}^d)^k \rightarrow \mathbb{R}$ where $g_2$ is non-zero and does not divide $g_1$.
Great care needs to be taken here however, since the map $g$ can now have problematic singularities that need to be dealt with in a sensible way.

One particularly interesting example of a rational $g$ map is the map
\begin{equation*}
    g: (\mathbb{R}^d)^k \dashrightarrow \mathbb{R}, ~ (p_1,\ldots,p_{d+1}) \mapsto \frac{-\det \Delta(p_1,\ldots,p_{d+1}) }{2 \det \Lambda(p_1,\ldots,p_{d+1})}
\end{equation*}
where (given $\mathbf{1} = ( 1,\ldots,1)$)
\begin{align*}
    \Lambda(p_1,\ldots,p_{d+1}) := \left( \|p_i-p_j\|^2 \right)_{i,j \in [d+1]}, \qquad \Delta(p_1,\ldots,p_{d+1}) := 
    \begin{pmatrix}
        0 & \mathbf{1}^T \\
        \mathbf{1} & \Lambda(p_1,\ldots,p_{d+1})
    \end{pmatrix}.
\end{align*}
When the points $p_1,\ldots,p_{d+1}$ are affinely independent,
$g(p_1,\ldots,p_{d+1})$ is the square of the circumradius of the points $p_1,\ldots,p_{d+1}$,
i.e., the radius of the unique hypersphere containing the points $p_1,\ldots,p_{d+1}$.

\subsection{Improvements on typical normed space results}

In comparison to \Cref{ABS}, the pinned variant given by \Cref{pinnedmostnorms} contains a notably weaker asymptotic term.
We believe that this can be removed in the following sense.

\begin{conjecture}\label{conj:pinnedmostnorms}
 The following is true for most $d$-norms $\|\cdot\|$. 
 For any finite point set $P \subseteq \mathbb R^d$, there exists a point $x\in P$, which determines $(1-o(1))|P|$ distances to the other points of $P$ with respect to $\|\cdot\|$,
 where the rate of decay $o(1)$ depends only on the norm $\|\cdot\|$.
\end{conjecture}

If \Cref{conj:pinnedmostnorms} is true,
then we can prove the following stronger version of \Cref{mainNormedresult} with relatively little effort.

\begin{prop}
    Suppose that \Cref{conj:pinnedmostnorms} is true.
    Then following is true for most $d$-norms $\|\cdot\|$. 
    Let $G = (V,E)$ be a connected graph, and let $P$ be any finite set of points in $\mathbb R^d$.
    Then 
    \begin{equation*}
       \Big|f_{\|\cdot\|, G}\left(P^{V}\right)\Big| \gg |P|^{|V|-1},
    \end{equation*}
    where the explicit constant depends only on the graph $G$ and the norm $\|\cdot\|$.
\end{prop}

\begin{proof}
    Set $f(n)$ to be the explicit function defining the rate of decay for the $o(1)$ term in \Cref{conj:pinnedmostnorms}.
    Following the same method given in \Cref{mainNormedresult} with $H(n) := (1-f(n))n$,
    we get from \cref{eq:mainNormedresult} that 
    \begin{equation*}
        \left| f_{\|\cdot\|,G}(P^V) \right| \geq (H(|P|/2^t)^{|V|-1} = \frac{(1-f(|P|))^{|V|-1}|}{2^{t(|V|-1)}}P|^{|V|-1} \gg |P|^{|V|-1}
    \end{equation*}
    as claimed.
\end{proof}

\subsection{Flexible graphs}

When dealing with anything of the same flavour as the Erd\H{o}s distinct distance problem, whatever happens on the grid is usually the `correct' answer.
In view of this ethos and \Cref{prop:genusgrid}, we conjecture the following.

\begin{conjecture}\label{conj:genus}
    Let $G=(V,E)$ be a fixed graph with 1 d.o.f.~in $\mathbb{R}^2$ and flex-genus at least 2.
    Then, for any finite set $P \subset \mathbb{R}^2$, we have
    \[ \Big|f_G\left(P^V\right) \Big| \geq |P|^{|V|-1}.\]
\end{conjecture}

The genus requirement described in \Cref{conj:genus} is needed. For example, if $G$ is a hinge, then it is 1 d.o.f.~graph with flex-genus 0,
but for any $\varepsilon > 0$, the equality holds $\big|f_G(P^V)\big| = O_{|P|}( |P|^{|V|-1-\varepsilon})$ when $P$ is a grid.

The only graphs left to then deal with are those with flex-genus 0 and 1.
We summarise our thoughts on the behaviour of $\big|f_G(P^V)\big|$ for each case below:
\begin{itemize}
    \item Flex-genus 0: we expect that the graph comprises of two graphs that are rigid in $\mathbb{R}^2$ joined at a vertex. If one of these rigid components is a single edge (equivalently, the original graph had a leaf vertex) then we know that we have logarithmic loss. Otherwise, we would actual expect no logarithmic loss.
    \item Flex-genus 1: each algebraic set $C_{G,vw}(p)$ when $p$ is contained in the integer grid is now an elliptic curve.
    It was conjectured by Goldfeld \cite{goldfeld} that a positive proportion of elliptic curves have only finitely many points.
    It is possible that one could leverage this to show that enough realisations on the grid have few equivalent realisations, which would then imply no logarithmic loss.
\end{itemize}
Because of this, we would expect that for 1 d.o.f.~graphs, the only criteria for having a logarithmic loss is the existence of a leaf vertex.
In fact, we would be even more bold to suggest the following:

\begin{conjecture}
    Let $G=(V,E)$ be a fixed connected graph with minimum degree 2.
    Then for any finite set $P \subset \mathbb{R}^2$, we have
    \[ \Big|f_G\left(P^V\right) \Big| \gg |P|^{|V|-1}.\]
\end{conjecture}

\section*{Acknowledgements}

This project arose from a focussed research group funded by the Heilbronn Institute for Mathematical Research (HIMR) and the UKRI/EPSRC Additional Funding Programme for Mathematical Sciences.
S.\,D.\ and J.\,P. are supported by HIMR.
N.F. was partially supported by ERC Advanced Grant ``GeoScape".
A.\,N.\ was partially supported by EPSRC grant EP/X036723/1.  A.W. is partially supported by FWF project PAT2559123.

\bibliography{realisationsinpointsets}

\appendix 

\section{Flex-genus is well-defined}\label{appendix}

In this section we prove the following technical lemma:

\Genus*

We break \Cref{lem:generalfibre} into two smaller results that are easier to prove.
The first we prove is the following:

\begin{lemma}\label{lem:app1}
    Let $f : \mathbb{C}^{m} \rightarrow \mathbb{C}^n$ be a dominant polynomial map.
    Then there exist smooth projective varieties $X,Y$, a non-empty Zariski open set $U \subset \mathbb{C}^{n}$ and a homogenous polynomial map $\tilde{f} :X \rightarrow Y$ such that the following holds:\setlength{\parskip}{0pt}
    \begin{enumerate}[label=(\roman*)]
        \item $\mathbb{C}^{m}$ is a Zariski open subset of $X$ and $\mathbb{C}^{n}$ is a Zariski open subset of $Y$;
        \item $f$ is exactly the map formed from restricting the domain and codomain of $\tilde{f}$ to $\mathbb{C}^{m}$ and $\mathbb{C}^{n}$ respectively;
        \item for any $\lambda,\mu \in U$, the fibres $\tilde{f}^{-1}(\lambda)$ and $\tilde{f}^{-1}(\mu)$ are homeomorphic smooth manifolds.
    \end{enumerate}
\end{lemma}

We will require the following theorem regarding when a map is a \emph{locally trivial fibration}:
a continuous surjective map $\pi : X \rightarrow B$ where every fibre $\pi^{-1}(b)$ is homeomorphic to some space $F$,
and for every point $b \in B$ there exists an open neighbourhood $U \subset B$ of $b$ and a homeomorphism $\phi : U \times F \rightarrow \pi^{-1}(U)$ such that $\pi \circ \phi (u,f) = u$ for all $u \in U$ and $f \in F$.
Note that in this theorem, a map is \textit{proper} if preimages of compact sets are compact.

\begin{theorem}[Ehresmann's Fibration Theorem, \cite{Ehresmann}]\label{thm:Ehresmann}
    Let $M,N$ be smooth manifolds, and $f:M \rightarrow N$ a proper surjective submersion, then $f$ is a locally trivial fibration.
\end{theorem}

\begin{proof}[Proof of \Cref{lem:app1}]
    Take $X$ to be the Zariski closure of the set $\graph(f):=\{(x,y)\in \mathbb{C}^{m} \times \mathbb{C}^{n} : y=f(x) \}$ in $\mathbb P^m \times \mathbb P^n$ (here the Cartesian product can be given via the Segre embedding; we will refer to points in $\mathbb P^m \times \mathbb P^n$ as pairs $(x,y)$ for simplicity).
    We observe that the map
    \begin{equation*}
        \phi : \mathbb{C}^{m} \rightarrow \graph(f), ~ x \mapsto (x, f(x))
    \end{equation*}
    is an isomorphism, and hence $\mathbb{C}^{m}$ is embedded in $X$. 
    By construction, the variety $X$ is both smooth and irreducible.
    Set $Y = \mathbb{P}^n$. With this, we define the projection
    \begin{equation*}
        \tilde{f} : X \rightarrow Y, ~ (x,y) \mapsto y.
    \end{equation*}
    With this, we note that $f = \tilde{f} \circ \phi$.
    By identifying $\graph(f)$ with $\mathbb{C}^m$ using the isomorphism $\phi$,
    we easily see that $f$ is the restriction of $\tilde{f}$ to the domain $\mathbb{C}^m$ and codomain $\mathbb{C}^n$.
    Since all complete varieties are compact spaces with respect to the analytic topology, and since $X$ is a closed subset of a complete variety, it follows that $\tilde{f}$ is a proper map. 
    
    Let $C$ be the Zariski closure of the critical values of $\tilde{f}$, and let $O$ be a dense open subset of $Y$ which is contained within the image of $\tilde{f}$, which exists due to \cite[Theorem 10.19]{GortzWedhorn}.
    As $f$ is dominant, so too is the map $\tilde{f}$,
    and so $\dim O = n$.
    By Sard's theorem, the set $C$ is a proper algebraic subset of $Y$,
    and so the set $\tilde{U} := O \setminus C$ is a non-empty Zariski open subset of $Y$.
    By restricting the domain of $\tilde{f}$ to the non-empty Zariski open set
    \begin{equation*}
        Z := X \setminus \left(\tilde{f}^{-1}\big((Y \setminus O) \cup C \big) \right)
    \end{equation*}
    and the codomain to $\tilde{U}$,
    we find a proper surjective submersion $h : Z  \rightarrow \tilde{U}$.
    Therefore by \Cref{thm:Ehresmann} this map is a locally trivial fibration.
    The locally trivial fibration property then directly implies that for every $\lambda, \mu$ contained in the same connected component of $\tilde{U}$, the fibres $h(\lambda) = \tilde{f}^{-1}(\lambda)$ and $h(\mu) = \tilde{f}^{-1}(\mu)$ are homeomorphic. However since $\tilde{U}$ is a Zariski open subset of projective space, it is connected in the analytic topology; see \Cref{lem:opensetsconnected}.
    We conclude the proof by setting $U = \tilde{U} \cap \mathbb{C}^m$.
\end{proof}

\begin{lemma}\label{lem:app2}
    Let $f : \mathbb{C}^{m} \rightarrow \mathbb{C}^n$ be a dominant polynomial map.
    Then there exists a non-empty Zariski open set $U \subset \mathbb{C}^{n}$ and an integer $k$ such that for any $\lambda \in U$, the fibre $f^{-1}(\lambda)$ consists of $k$ smooth manifolds,
    each of dimension $m-n$.
\end{lemma}

\begin{proof}
    We first construct an open set in $\mathbb{C}^n$ such that each fibre $f^{-1}(\lambda)$ of the set is a smooth $(m-n)$-dimensional manifold.
    Take $C \subset \mathbb{C}^n$ be the Zariski closure of the critical values of $f$.
    The set $U_1 := \mathbb{C}^n \setminus C$ is a Zariski open set which is non-empty due to $f$ being dominant and Sard's theorem.
    It now follows from the constant rank theorem that each irreducible component of the fibre $f^{-1}(y)$ is a smooth $(m-n)$-dimensional manifold.
    
    We now turn to scheme theory to prove the remainder of the result.
    Fix the irreducible and reduced affine schemes $X = \Spec ( \mathbb{C}[x_1,\ldots,x_m])$ and $Y = \Spec ( \mathbb{C}[y_1,\ldots,y_n])$,
    and fix $h:X \rightarrow Y$ to be the domain and codomain extension\footnote{This can be achieved by setting $h(\mathfrak{p}) := (f^*)^{-1}(\mathfrak{p})$ for every prime ideal $\mathfrak{p}$ of $\mathbb{C}[x_1,\ldots,x_m]$, where $f^* :\mathbb{C}[y_1,\ldots,y_n] \rightarrow \mathbb{C}[x_1,\ldots,x_m]$ is the $\mathbb{C}$-algebra homomorphism that maps $g$ to $f \circ g$.} of $f$.
    In the language of schemes, $h$ is a morphism of finite type between affine schemes.
    It is known that for any point $\mu \in Y$, there exists an open set $O \subset \overline{\{\mu\}}$ such that the number of connect components of each fibre $h^{-1}(\lambda)$ with $\lambda \in O$ is constant; see \cite[\href{https://stacks.math.columbia.edu/tag/055H}{Lemma 055H}]{stacksproject}.
    By taking $\mu$ to be the generic point\footnote{An affine scheme $\Spec(R)$ is irreducible and reduced if and only if the set $\{0\}$ is a prime ideal of the ring $R$. Since the closure of $\{0\}$ is the entirety of $\Spec(R)$, we say in this case that $\{0\}$ is the \emph{generic point} of $\Spec (R)$.} of $Y$,
    we see that there exists an open set $\tilde{U}_2 \subset Y$ and $k \in \{0,1,\ldots, \infty\}$ such that each fibre $h^{-1}(\lambda)$ with $\lambda \in \tilde{U}_2$ has $k$ connected components. 
    Now fix $U_2$ to be the closed points contained in $\tilde{U}_2$.
    Since $\mathbb{C}^n$ with the Zariski topology is exactly the subspace of $Y$ formed by restricting to the closed points,
    the set $U_2$ is a Zariski open subset of $\mathbb{C}^n$.

    Now take $U := U_1 \cap U_2$ and choose any $\lambda \in U$.
    We note here that the fibre $f^{-1}(\lambda)$ is exactly the set of closed points in $h^{-1}(\lambda)$.
    Choose any connected component $C$ of $h^{-1}(\lambda)$ and let $C'$ be the closed points of $C$.
    We now take note of the following two facts:
    (i) an algebraic variety $X \subset \mathbb{C}^d$ is connected if and only if the coordinate ring $\mathbb{C}[X]$ has no non-trivial (i.e., neither 0 nor 1) idempotent elements;
    (ii) an affine scheme $\Spec (R)$ is connected if and only if the ring $R$ has no non-trivial idempotent elements.
    Hence $C$ is connected.
    It now follows that the fibre $f^{-1}(\lambda)$ contains exactly $k$ connected components, each of which being a $(m-n)$-dimensional smooth manifold.  
\end{proof}

We also require the following result.

\begin{prop}[see {\cite[Proposition 38]{lubbes2024irreduciblecomponentssetspoints}}]\label{prop:app3}
    Let $f : \mathbb{C}^{n+1} \rightarrow \mathbb{C}^n$ be a dominant polynomial map.
    Then there exists a non-empty Zariski open set $U \subset \mathbb{C}^{n}$ such that the following holds:
    for any $\lambda \in U$, any irreducible component of $f^{-1}(\lambda)$ is a smooth algebraic curve;
    moreover, any two irreducible components of $f^{-1}(\lambda)$ have the same genus.
\end{prop}

With this, we are now ready to prove the key lemma of the section.

\begin{proof}[Proof of \Cref{lem:generalfibre}]
    Take $U_1,U_2,U_3 \subset \mathbb{C}^n$ to be the non-empty Zariski open sets guaranteed by \Cref{lem:app1}, \Cref{lem:app2} and \Cref{prop:app3} respectively,
    and take $U_4 \subset \mathbb{C}^{n+1}$ to be the Zariski open set of points for which the rank of the Jacobian of $f$ is $n$.
    Now set $U := f^{-1}(U_1 \cap U_2 \cap U_3) \cap U_4$.
    It is clear that for each $p \in U$, the Jacobian ${\rm J} f(p)$ has rank $n$ and the fibre $f^{-1}(f(p))$ consists of $k$ disjoint smooth algebraic curves.
    Now choose any $p \in U$.
    Each connected component of $f^{-1}(f(p))$ is a Zariski open subset of a connected component of $\tilde{f}^{-1}(f(p))$.
    Hence by \Cref{prop:app3}, all the connected components of $f^{-1}(f(p))$ have the same genus $g$.
    If we choose some other point $q \in U$, then it follows from \Cref{lem:app1} that all connected components of $f^{-1}(f(q))$ also have genus $g$.
    This now concludes the proof.
\end{proof}

Finally, we record a proof that Zariski open subsets of complex projective spaces are connected with the analytic topology.

\begin{lemma} \label{lem:opensetsconnected}
        Let $U \subseteq \mathbb P^n$ be a Zariski open set. Then $U$ is connected in the analytic topology.
    \end{lemma}
    
    \begin{proof}
    We assume $U$ is non-empty and not $\mathbb P^n$ itself, otherwise the result is trivial. We now use induction. For the base case, $U \subseteq \mathbb P^1$ is the complement of the zeroes of a homogeneous polynomial $f \in \mathbb C[x,y]$, and is therefore simply $\mathbb P^1$ with finitely many points removed, which is connected (one can see this by picking charts so that $U$ is equivalent to $\mathbb R^2$ with finitely many points removed).

    Our induction hypothesis is that any open subset $U \subseteq \mathbb P^k$ is analytically connected. Now take any open subset $U \subseteq \mathbb P^{k+1}$. Suppose that $U$ is not connected, and write $U = U_1 \sqcup U_2$. As $U$ is Zariski open, it can be written as $U = \mathbb P^{k+1} \setminus Z(f_1,...,f_t)$, where $Z(f_1,...,f_t) \subset \mathbb P^{k+1}$ is the zero locus of homogeneous polynomials $f_1,...,f_t \in \mathbb C[x_1,...,x_{k+1}]$. Now take any two points $p \in U_1$ and $q \in U_2$, and take any hyperplane $H$ containing both of these points. By making a change of basis, we can assume that this hyperplane is given by $x_{k+1}=0$. Now consider the set $V :=H \cap U$. This is an open subset of $\mathbb P^{k}$ as it can be written as $V = \mathbb P^{k} \setminus Z(f_1',...,f_t')$, where $f_i'(x_1,...,x_k) := f_i(x_1,...,x_k,0)$. Furthermore, $V$ is not analytically connected, since it can be written as 
    $$V = H \cap U = H \cap (U_1 \sqcup U_2) = (H\cap U_1) \sqcup (H \cap U_2)$$
    and so $V$ can be written as the union of two analytically open sets via the induced topology (note that the two sets are not empty as $p \in H\cap U_1$ and $q \in H \cap U_2$). This contradicts our inductive hypothesis, finishing the proof.
    \end{proof}

\end{document}